\documentclass[10pt]{article}
\usepackage[utf8]{inputenc}
\usepackage{amsmath, amsthm, amssymb, bm, txfonts, yhmath}
\usepackage[T1]{fontenc}        
\usepackage{lmodern}         
\usepackage[english]{babel}
\usepackage{indentfirst}
\usepackage{graphicx}
\usepackage{titlesec}
\usepackage{caption}
\usepackage{subcaption}
\usepackage{listings}
\usepackage{xcolor}
\usepackage{comment}
\usepackage[nottoc]{tocbibind}
\usepackage{hyperref}
\usepackage{pdfpages}
\usepackage{url}
\usepackage{authblk}
\usepackage{arydshln}
\usepackage{csquotes}

\newcommand{\RR}{\mathbb{R}}

\newcommand{\NN}{\mathbb{N}}

\newtheorem{theorem}{Theorem}
\newtheorem{lemma}{Lemma}

\newtheorem{defi}{Definition}

\newcounter{example}
\newenvironment{example}[1][]{\refstepcounter{example}\par\medskip
   \noindent \textbf{Example~\theexample. #1} \rmfamily}{\medskip}
\newtheorem{remark}{Remark}

\newcounter{algorithm}
\newenvironment{algorithm}[1][]{\refstepcounter{algorithm}\par\medskip
   \noindent \textbf{Algorithm~\thealgorithm. #1} \rmfamily}{\medskip}

\title{Identifiability and Observability Analysis for Epidemiological Models: Insights on the SIRS Model}
\author[1]{Alicja B. Kubik\thanks{akubik@ucm.es}}
\author[1]{Benjamin Ivorra\thanks{ivorra@mat.ucm.es}}
\author[2]{Alain Rapaport\thanks{alain.rapaport@inrae.fr}}
\author[1]{Ángel M. Ramos\thanks{angel@mat.ucm.es}}
\affil[1]{MOMAT, Instituto de Matem\'atica Interdisciplinar (IMI), Univ. Complutense de Madrid, 28040 Madrid, Spain}
\affil[2]{MISTEA, Univ. Montpellier, INRAE, Institut Agro, 34060 Montpellier, France}

\date{}

\begin{document}

\maketitle

\begin{abstract}
The problems of observability and identifiability have been of great interest as previous steps to estimating parameters and initial conditions of dynamical systems to which some known data (observations) are associated. While most works focus on linear and polynomial/rational systems of ODEs, general nonlinear systems have received far less attention and, to the best of our knowledge, no general constructive methodology has been proposed to assess and guarantee parameter and state recoverability in this context. We consider a class of systems of parameterized nonlinear ODEs and some observations, and study if a system of this class is observable, identifiable or jointly observable-identifiable; our goal is to identify its parameters and/or reconstruct the initial condition from the data. To achieve this, we introduce a family of efficient and fully constructive procedures that allow recoverability of the unknowns with low computational cost and address the aforementioned gap. Each procedure is tailored to different observational scenarios and based on the resolution of linear systems. As a case study, we apply these procedures to several epidemic models, with a detailed focus on the SIRS model, demonstrating its joined observability-identifiability when only a portion of the infected individuals is measured, a scenario that has not been studied before. In contrast, for the same observations, the SIR model is observable and identifiable, but not jointly observable-identifiable. This distinction allows us to introduce a novel approach to discriminating between different epidemiological models (SIR vs. SIRS) from short-time data. For these two models, we illustrate the theoretical results through some numerical experiments, validating the approach and highlighting its practical applicability to real-world scenarios.
\end{abstract}

{\small \noindent\textit{Key words:}
Nonlinear epidemiological models; Parameter identifiability; State observability; Joined observability-identifiability; SIRS model analysis; Epidemiological model discrimination; Epidemic data reconstruction.}

\section{Introduction}\label{intro}

Mathematical modeling has been widely used for studying epidemics for a long time. Some of the most popular models in epidemics generated by infectious diseases are compartmental models, whose theory was set up in the early 1900s (see \cite{KMK}), and have been regularly used since then to estimate the dynamics of different diseases (for example, influenza \cite{InfluenzaBrauer, InfluenzaChen}, tuberculosis \cite{TBCastillo, TBJiang}, Ebola \cite{EbolaGumel,EbolaSeck}, or vector-transmitted diseases such as malaria \cite{malaria}, dengue fever or the Zika virus \cite{DengueBrauer,DengueFeng}), with a significant surge in applications in recent years due to the COVID-19 pandemic (see, for example, \cite{CovidArino,Bertozzi, China,Quarteroni,VyV}). These models are usually deterministic systems of nonlinear ODEs, which include the class of systems that we are going to address from now on.

Given a disease and some information on the different biological processes that are involved, we need to set up a model that captures the most important features of its dynamics, being careful not to complexify it up to a point that may result intractable from a mathematical perspective. To these dynamical models there are associated data, available in the literature (e.g., parameter values) or collected over time (e.g., how many people are hospitalized or vaccinated at certain dates). This information is partial most of the time, in the sense that they are related to some of the parameters and state variables of the model, but rarely to all of them. For instance, the sizes of susceptible or asymptomatic compartments are often hard or even impossible to measure (they are hidden variables), and the transmission rate of a new virus is often not known in advance.

Once we have settled a model and know some information about the parameters and compartments, a question arises: Can we uniquely reconstruct all the unknowns of the model from partial observations? The ability to accurately reconstruct key aspects of disease dynamics from observed data is fundamental to understand epidemic trajectories and design effective control strategies. Thus, determining whether the parameters can be uniquely identified or the states can be precisely reconstructed from available data is essential. We thus face an inverse problem. Such identification and state reconstruction problems are often met in automatic control for linear and nonlinear dynamics. Although many successful applications have been recorded in aeronautics, automotive, electronics, or chemistry, among other domains, very few has been investigated for epidemiological models. Moreover, despite substantial research addressing these questions, significant challenges remain, particularly in nonlinear models, where exact reconstruction of parameters and states is often elusive and requires sophisticated analytical approaches. 

In this context, the determination of a feasible set of parameters and an initial condition that align with observed data falls into two key problems: The \textit{parameter identification problem} and the \textit{state observation problem} in automatic control literature. These have been addressed in many different ways (see, for example, \cite{IdentCapistran, IdentChowell, IdentGoodwin, ObserverHuong, ObsCOVID, EstimationMedvedeva, IdentPronzato}). However, before attempting to determine these unknowns, it is crucial to assess the extent of recoverable information from the available epidemic data. For example, can we determine the disease contact rate if we know the data of hospitalized people? What about the loss of immunity rate? How many people were infected when the data started to be reported? Of course, these are not issues related only to epidemiological models. In general, given a phenomenon we want to model (e.g., physical, biological, or mechanical), we will have some parameters, a \textit{state vector} (in the epidemiological case, these states are the size of the population in each compartment, or a portion of this size), and some observed data of this phenomenon (also called \textit{measurements}, \textit{outputs} of the model, or \textit{signals} \cite{TermKalman}). The theories of \textit{identifiability} and \textit{observability} provide the foundation for addressing these questions. Let us briefly introduce these concepts.

Identifiability theory is the one in charge of deciding, given an initial condition and some observed data, if we can determine univocally all the unknown parameters that govern our system. If it is not possible, we can possibly recover them \textit{partially} (some of them or some combination of them).

Observability theory is the one in charge of deciding, given the parameters and some observed data, if we can recover the initial condition; in other words, if there exists a unique initial condition such that the solution starting from this initial condition matches our observations, and hence we can distinguish different states from partial measurements of the state vector.

If one is interested in both identifying and observing the system, i.e., if it is \textit{jointly observable-identifiable} (see \cite{Cunniffeetal}) it is common to extend the system by considering the parameters as part of the state vector (with a null dynamics), and then considering the observation problem of this extended system. The most commonly used technique to decide whether a system is observable or not is the Hermann-Krener condition, or Observability Rank Condition, that consists in studying the separability of a set composed by the observations and its Lie derivatives with respect to the vector field of the system (see \cite{ORC}); this computation may be sometimes simplified by exploiting symmetries or groups of invariance of the system (in particular, for mechanical and/or robotics systems, see e.g. \cite{Martinelli}). If we conclude our system is observable, then we can make use of several techniques that may help us estimate practically the unknowns; in particular, to perform this task, we can try to construct a \textit{state observer}, as for example the Luenberger observer \cite{Luenberger}, high-gains observers \cite{HighGains}, or the Kalman filter \cite{Kalman}. These techniques aim to \textit{estimate} the state vector at a certain speed with a certain accuracy, minimizing the estimation error at a given time, such that it is usually asymptotically null, but do not guarantee exact reconstruction. Sometimes, however, one can treat the system algebraically and try to reconstruct exactly the unknowns in terms of, for example, the derivatives of the data whenever they exist and are known (or can be computed) perfectly (see \cite[Chapter 3]{Cunniffeetal}).

The theories of identifiability and observability extend beyond the epidemiological context and are relevant to any phenomena modeled through ODE systems with measurable outputs or functions of the states (see \cite{bioreactorsBastin, vehicleBicchi, orbitalGeller, JCD2024, Jain2019, ObsAutomatica}). Indeed, there is few literature about studying the joined observability-identifiability of epidemiological models (see \cite{Cunniffeetal}, which is a survey that can be used as a textbook about this topic, and \cite{IdentObsFang}, \cite{RapaportQuarantine}, and \cite{IdentObsVillaverde}).

Several works have emphasized the importance of initial conditions for identifying the parameters of a general system, and not every initial condition is suitable, i.e., some initial conditions can produce the same output when considering different parameters (see, e.g., \cite{VidalBlanchard2001}, \cite{Saccomani2003}). While there are fields where one can perform different experiments with chosen initial conditions in order to identify the parameters (see, e.g., \cite{VidalBlanchard2001, Jeronimo2019}), in disciplines such as Epidemiology this is not possible, and hence we need to study which initial conditions are not suitable, and if we can avoid them. In this paper, we build on the theory developed in \cite{Kubik}, and extend the results by considering the existence of non-suitable initial conditions which are dependent on the parameters. 

On the other hand, most of the research carried out in observability and identifiability of nonlinear dynamics is centered on polynomial and rational systems of equations (see, e.g., \cite{Fliess1981, Gevers2016, Ljung1994, Ljung1990, Saccomani2003}), and the study of nonlinear, non-rational systems is scarce (see the survey \cite{VidalAnstett2020}). In rational systems, the goal is typically finding the \textit{input-output} (IO) \textit{equations} in order to determine if the parameters are identifiable. These are structural equations of the system under study, which relate the parameters, the outputs and their derivatives, and their inputs (if any). However, {we} consider that, in these works, there lacks a direct, unified argument to conclude that these parameters can be effectively recovered from these equations. In \cite{Kubik}, general nonlinearities {are considered}, and {it advances the state of the art by minimizing reliance on high-order derivatives, improving robustness with respect to noisy and sparse datasets, and offering novel techniques to address joined observability-identifiability in nonlinear systems of ODEs.} Moreover, this lacking argument {is also covered} by considering an assumption about linear independence of functions of the observations and their derivatives. This assumption of linear independence was already tackled in \cite{VidalBlanchard2001} to prove identifiability alone, however, it is framed in the case of rational systems and it assumes knowledge of initial conditions and having data at the initial time; these last assumptions are not feasible in some fields such as Epidemiology. It is as well mentioned in \cite{Ovchinnikov2022} for linear ODE systems, which do not cover epidemiological models either.

This paper is structured as follows. In Section \ref{motivation}, we motivate through the SIRS model the application of this framework to epidemiological models. Then, in Section \ref{gen_model}, we revisit and extend the theoretical part presented in \cite{Kubik}, and establish different constructive algorithms to recover the parameters and/or initial condition. We also extend the theory to incorporate higher-order derivatives, and highlight that this framework can be extended to scenarios with piecewise constant parameters, provided that the time instants of parameter changes are known or the time intervals with constant values are large enough.

In Section \ref{obsnident}, we validate this approach by applying it to the SIRS model, proving joined observability-identifiability under partial observations of the infected compartment, an example that we have not seen considered in the literature. We will observe that in this case it is crucial considering initial conditions in parameter-dependent sets. Furthermore, although this is a rational model which can be studied through classical techniques, using our methodology allows for early-stage model discrimination between SIR and SIRS structures, enabling the detection of immunity loss directly from observational data. This distinction addresses a key gap in the literature, to the best of our knowledge, where theoretical studies often overlook model selection based on limited data. After this, we present in Section \ref{examples} other epidemiological examples to highlight the interest of the proposed approach. Finally, we perform in Section \ref{numerical} different numerical experiments using the SIRS and the SIR models in two different scenarios that illustrate our approach.

\section{Motivating framework: Insights from the SIRS model} \label{motivation}

Compartmental models, such as the SIRS model, strongly depend on real data to determine their parameters and the current state of the disease. These data are typically scarce and noisy, making it crucial to develop methods to estimate the unknown parameters and/or initial condition from partial observations. For this, we are going to follow two different approaches, depending on the type of available data: identifiability and observability, both formally defined in Section \ref{gen_model}.

Let us consider the following classical SIRS model together with an output, given by an unknown fraction $k\in(0,1]$ of the infected individuals at time $t\geq 0$, represented by $y$:
\begin{equation}\label{motiSIRSkI}\left\{\begin{array}{ccl}
    \dot{S}&=&-\beta SI+\mu R,\\
    \dot{I}&=&\beta SI-\gamma I, \\
    \dot{R}&=&\gamma I-\mu R, \\
    y&=& kI,
\end{array} \quad \forall\, t\geq 0.\right.\end{equation}

This model is particularly relevant for understanding the dynamics of diseases such as influenza of some coronaviruses. It is also of interest in different realistic situations, such as performing random tests among the population, or considering that a fraction $1-k$ of the infectious people are asymptomatic and hence we do not measure those cases, where the parameter $k$ is typically unknown. For instance, during the COVID-19 pandemic, one of the main challenges when studying the prevalence of the disease was estimating the asymptomatic cases, whereas governments also conducted randomized testing to estimate this prevalence (see, for example, \cite{Italy}).

Given this model, we are interested in knowing whether it is identifiable, observable or jointly observable-identifiable. To address this, in Section \ref{gen_model}, we revisit and extend the general framework of \cite{Kubik}. The presented SIRS model with output $kI$ fits within the class of systems presented in Section \ref{gen_model} and serves as a motivating example. Indeed, in Section \ref{obsnident} we make use of this approach to prove its joined observability-identifiability when $\mu>0$. Moreover, we compare these results to those obtained for the (\textit{a priori} simpler) SIR model, which can be considered as a limiting case of the SIRS model with $\mu=0$, considering the same observational data (i.e., a fraction of infected individuals). Our analysis reveals that the SIR model is both observable and identifiable, but lacks joined observability-identifiability, highlighting a key distinction between these two models that allows us to distinguish them in an early stage. We will illustrate this difference through the numerical tests presented in Section \ref{numerical}.

\section{A general framework} \label{gen_model}

In this section, we extend the general framework established in \cite{Kubik} for analyzing observability, identifiability and joined observability-identifiability of dynamical systems, particularly focusing on systems modeled by autonomous differential equations with unknown parameters. We provide conditions under which the parameters and/or initial condition of such systems can be uniquely determined from observations of the system’s output. To this end, we first introduce the mathematical formulation of the model and provide formal definitions of the key concepts: Identifiability, observability, and joined observability-identifiability. Unlike in \cite{Kubik}, when discussing identifiability, we will consider the possibility of parameter-dependent initial conditions that impede identifiability; this is a typical situation in epidemiological models. Additionally, we revisit some concepts on Lie derivatives that are instrumental in the theoretical analysis. Building on this foundation, we extend the key results from \cite{Kubik}, taking into account these parameter-dependent initial conditions, that enable to establish different constructive methodologies to recover the initial condition and/or the parameters, whenever possible. Moreover, we present a third approach using higher-order derivatives to recover the unknowns in cases where the system is proven to be jointly observable-identifiable.

\subsection{Mathematical formulation and definitions}

Consider a phenomenon that can be modeled by a system of first order autonomous differential equations, which depends on some unknown parameter vector $\theta$. Together with an initial condition $\xi$, we assume the model provides some \textit{output} $y_{(\xi,\theta)}$, described by a suitable function $h$. The mathematical formulation of the model and its output is given by
\begin{equation}\label{general_sys}
    \left\{\begin{array}{l}
    \dfrac{\mathrm{d} x}{\mathrm{d} t}(t;\xi,\theta)=f(x(t;\xi,\theta),\theta), \quad x(0;\xi,\theta)=\xi, \\[1em] y_{(\xi,\theta)}(t)=h(x(t;\xi,\theta),\theta),
    \end{array}\right.
\end{equation}
where $f(\cdot , \cdot): \Omega\times \Theta\rightarrow \mathbb{R}^n$ is a known function of $(x,\theta)$ which is locally Lipschitz-continuous w.r.t. $x\in\Omega\subset\mathbb{R}^n$ (to guarantee uniqueness of solutions) and continuous w.r.t. $\theta\in\Theta\subset\mathbb{R}^b$; $\theta$ are the constant parameters of the system; $\Omega\subset\mathbb{R}^n$ is a positively invariant set with respect to the system of ODEs of System \eqref{general_sys}, for any $\theta\in\Theta$ (i.e., solutions starting in $\Omega$ remain in $\Omega$ during all its definition time); $x(\cdot ; \xi,\theta):\mathcal{I}\rightarrow\Omega\subset\mathbb{R}^n$ denotes the unique solution of the system of ODEs of System \eqref{general_sys} with initial condition $\xi\in \Omega$ and we assume it is globally defined, i.e., $\mathcal{I}=[0,+\infty)$; and the output $y_{(\xi,\theta)}(t)$, $t\in\mathcal{S}\subset\mathcal{I}$, is described by some known function $h( \cdot , \cdot ):\Omega\times\Theta\rightarrow\mathbb{R}^m$.

We aim to answer one question: given the output $y_{(x_0,\theta_0)}(\cdot)$ for certain $(x_0,\theta_0)\in\Omega\times\Theta$, can we uniquely determine $x_0$ (given $\theta_0$), $\theta_0$ (given $x_0$) or both of them?

We present now formal definitions of the identifiability and observability concepts mentioned in the Introduction. As previously mentioned, some parameter-dependent initial conditions may not allow to distinguish different parameters. Hence, we will consider, for each $\theta\in\Theta$, a subset $\Omega_{\theta}\subset\Omega$ positively invariant w.r.t. the ODE system in System \eqref{general_sys}{, and the set $\Gamma_{\Theta}=\bigcup_{\theta\in\Theta}\Omega_{\theta}\times\{\theta\}$. Notice that, if $\Omega_{\theta}=\Omega$, for each $\theta\in\Theta$, then $\Gamma_{\Theta}=\Omega\times\Theta$.}

\begin{defi}[Identifiability in a time set]\label{identifiability_set}
    System \eqref{general_sys} is identifiable on $\Theta$ in $\mathcal{S}\subset\mathcal{I}$ {with initial conditions consistent with} the family $\{\Omega_{\theta}\}_{\theta\in\Theta}$ if, for any $\theta_0\in\Theta$, for any $\xi\in\Omega_{\theta_0}$, a different $\theta_0'\in\Theta$ produces a different output at some time $t\in\mathcal{S}$, i.e., $$h(x(t;\xi,\theta_0), \theta_0)\neq h(x(t;\xi,\theta_0'), \theta_0').$$
    Equivalently, if $h(x(t;\xi,\theta_0), \theta_0)=h(x(t;\xi,\theta_0'), \theta_0')$, for all $t\in\mathcal{S}$ and any $\xi\in\Omega_{\theta_0}$, then $\theta_0=\theta_0'$. {Moreover, if $\Omega_{\theta}=\Omega$, for each $\theta\in\Theta$, we say that System \eqref{general_sys} is identifiable on $\Theta$ in $\mathcal{S}\subset\mathcal{I}$ with initial conditions in $\Omega$.}
\end{defi}

A typical initial condition which may not allow identifiability are the equilibrium points that depend on the parameters, in the case that different parameter vectors generate the same constant observation. This phenomenon will be illustrated in Section \ref{obsnident}. 

\begin{defi}[Observability in a time set]\label{observability_set}
    System \eqref{general_sys} is observable on $\Omega$ in $\mathcal{S}\subset\mathcal{I}$ {with parameters in} $\Theta$ whether, for any $\theta\in\Theta$, any two different $x_0,x_0'\in\Omega$ produce a different output at some time $t\in\mathcal{S}$, i.e.,
    \begin{equation*}
    h(x(t;x_0,\theta), \theta)\neq h(x(t;x_0',\theta), \theta).\end{equation*}
    Equivalently, if $h(x(t;x_0,\theta), \theta)=h(x(t;x_0',\theta), \theta)$, for all $t\in\mathcal{S}$ and any $\theta\in\Theta$, then $x_0=x_0'$.
\end{defi}

\begin{defi}[Identifiability]\label{identifiability}
    System \eqref{general_sys} is identifiable on $\Theta$ {with initial conditions consistent with the family} $\{\Omega_{\theta}\}_{\theta\in\Theta}$ if Definition \ref{identifiability_set} is satisfied for $\mathcal{S}=\mathcal{I}$. {Moreover, if $\Omega_{\theta}=\Omega$, for each $\theta\in\Theta$, we say that System \eqref{general_sys} is identifiable on $\Theta$ with initial conditions in $\Omega$.}
\end{defi}

\begin{defi}[Observability]\label{observability}
    System \eqref{general_sys} is observable on $\Omega$ {with parameters in} $\Theta$ if Definition \ref{observability_set} is satisfied for $\mathcal{S}=\mathcal{I}$.
\end{defi}

Notice that identifiability (resp. observability) in a time set $\mathcal{S}$ implies identifiability (resp. observability)
in any time set $\mathcal{A}$ such that $\mathcal{S}\subset\mathcal{A}$. However, the converse is not necessarily true, as demonstrated by the following examples.

\begin{example}\label{ex1}
    Consider the following system: \begin{equation}\label{examplesys}\left\{\begin{array}{ccl}
        x'(t) & = & k, \quad x(0)=\xi, \\
        y_{(\xi,k)}(t) & = & \max(1,x(t)),
    \end{array}\right.\end{equation} where $k\in\Theta=(0,1]$ is unknown. Let $\Omega=\Omega_k=[0,\infty)$, $k\in\Theta$, be the positively invariant set w.r.t. the ODE given in \eqref{examplesys}, $\xi\in\Omega$, $\mathcal{I}=[0,\infty)$, and $f(x,k)=k$ (which is Lipschitz-continuous w.r.t. $x$ and continuous w.r.t. $k$). The unique solution of the ODE is $x(t)=kt+\xi$, $t\geq 0.$ The system is identifiable on $\Theta$ in $\mathcal{I}$ {with initial conditions in} $\Omega$, since, for any $\xi\in\Omega$, considering different $k_1,k_2\in\Theta$, we have that \begin{equation*}\max(1,k_1t+\xi)=k_1t+\xi\neq k_2t+\xi=\max(1,k_2t+\xi),\end{equation*}for all $t\geq 1/\min\{k_1,k_2\}$. However, if $\mathcal{S}=[0,1/2]\subset \mathcal{I}$, $\xi\in[0,1/2]$ and $k_1,k_2\in\Theta$, with $k_1\neq k_2$, then \begin{equation*}\max(1,k_1t+\xi)=1=\max(1,k_2t+\xi),\end{equation*}for any $t\in\mathcal{S}$, and hence the system is not identifiable on $\Theta$ in $\mathcal{S}$ {with initial conditions in} $\Omega$. \qed
\end{example}

\begin{example}
    We consider the same case as the one shown in Example \ref{ex1}. The system is observable on $\Omega$ in $\mathcal{I}$ {with parameters in} $\Theta$, since, for any $k\in\Theta$, considering different $x_0,x_0'\in\Omega$, we have that \begin{equation*}\max(1,kt+x_0)=kt+x_0\neq kt+x_0'=\max(1,kt+x_0'),\end{equation*}for all $t\geq 1/k$. However, if $\mathcal{S}=[0,1/2]\subset \mathcal{I}$ and $x_0,x_0'\in[0,1/2]\subset\Omega$, with $x_0\neq x_0'$, then \begin{equation*}\max(1,kt+x_0)=1=\max(1,kt+x_0'),\end{equation*}for any $k\in\Theta$, $t\in\mathcal{S}$, and hence the system is not observable on $\Omega$ in $\mathcal{S}$ {with parameters in} $\Theta$. \qed
\end{example}

A natural extension of this framework consists in treating the parameters as part of the states of an augmented system. If one extends the dynamics with $\dot{\theta}=0$, then both the identifiability and observability properties can be studied as a particular case of observability in higher dimension: If the extended system is observable, then the original system is both observable and identifiable. However, the reverse implication is not generally true. This is related to the \textit{joined observability-identifiability} of
a system (see \cite{Cunniffeetal, Tunali1987}).

We are now going to consider that we need to recover both $x_0$ {and $\theta_0$.}

\begin{defi}[Joined observability-identifiability in a time set]\label{joined_set}
    System \eqref{general_sys} is jointly observable-identifiable on $\Gamma_{\Theta}$ in $\mathcal{S}\subset\mathcal{I}$ if, for any $(x_0,\theta_0)\in\Gamma_{\Theta}$, a different $(x_0',\theta_0')\in\Omega\times\Theta$ produces a different output at some time $t\in\mathcal{S}$, i.e., $$h(x(t;x_0,\theta_0),\theta_0)\neq h(x(t;x_0',\theta_0'),\theta_0').$$
    Equivalently, if $h(x(t;x_0,\theta_0),\theta_0)= h(x(t;x_0',\theta_0'),\theta_0')$ for all $t\in\mathcal{S}$, this implies that $x_0=x_0'$ and $\theta_0=\theta_0'$.
\end{defi}

\begin{defi}[Joined observability-identifiability]\label{joined}
    System \eqref{general_sys} is jointly observable-identifiable on $\Gamma_{\Theta}$ if Definition \ref{joined_set} is satisfied for $\mathcal{S}=\mathcal{I}$.
\end{defi}

{Furth}ermore, joined observability-identifiability in a time set $\mathcal{S}$ implies joined observability-identifiability
in any time set $\mathcal{A}$ such that $\mathcal{S}\subset\mathcal{A}$.

Note that recovering the parameters and initial condition from the data up to time $t$ is equivalent to reconstructing the parameters and state at current time $t$, because the dynamics is deterministic and reversible.

Our main focus will be on joined observability-identifiability; however, we will present the results in a way that they may be applicable separately to observability or identifiability. In particular, as commented before, it is clear that joined observability-identifiability implies both observability and identifiability independently.

In the following, we recall the well-known concept of Lie derivative, which is commonly used for studying observability and identifiability. We briefly review it here for completeness.

\subsection{About Lie derivatives}

Let $\NN_0=\NN\cup\{0\}$. Consider $\varphi\in\mathcal{C}^{d}(\Omega; \mathbb{R}^m), \ F\in\mathcal{C}^{\max\{0,d-1\}}(\Omega; \mathbb{R}^n)$, for some $d\in\mathbb{N}\cup\{0\}$, and the map \begin{equation*}\begin{array}{lccl}\mathcal{L}_{F,\varphi,d}: & \Omega & \longrightarrow & \mathbb{R}^{m\times (d+1)} \\ & \xi & \longmapsto & \left(L_F^0\varphi(\xi),L_F^1\varphi(\xi),\dots,L^d_F\varphi(\xi)\right),\end{array}\end{equation*}where $L_F^0\varphi=\varphi$ and, if $d\geq1$, $L^1_F\varphi\in\mathcal{C}^{d-1}(\Omega;\ \mathbb{R}^m)$ is the Lie derivative of $\varphi$ with respect to the vector field $F$, i.e., for any $\xi\in\Omega$, \begin{equation*}L_F^1\varphi(\xi)=\mathrm{D}\varphi(\xi)F(\xi),\end{equation*}where $\mathrm{D}\varphi$ is the $m\times n$ Jacobian matrix of $\varphi$. In particular, if $z(\cdot;\xi)$ is a solution of \begin{equation}\label{system_z}\dot{z}(t;\xi)=F(z(t;\xi)),\quad z(0;\xi)=\xi,\end{equation}such that $\Omega$ is positively invariant with respect to this system, then, for any $t\geq 0$, \begin{equation}\label{LFt}L_F^1\varphi(z(t;\xi))=\mathrm{D}\varphi(z(t;\xi))F(z(t;\xi))=\dfrac{\mathrm{d} }{\mathrm{d} t}\varphi(z(t;\xi)).\end{equation}
Then,\begin{equation}\label{L1xi}
    L_F^1\varphi(\xi)=\left.\dfrac{\mathrm{d}}{\mathrm{d} t}\varphi(z(t;\xi))\right|_{t=0}.
\end{equation} If $d\geq2$, we also define $L^k_F\varphi\in\mathcal{C}^{d-k}(\Omega; \mathbb{R}^m)$, for $k\in\{2,\dots,d\}$, by recursion as follows: \begin{equation*}L^k_F\varphi=L_F^1\left(L^{k-1}_F\varphi\right).\end{equation*}

\begin{lemma}\label{Lxi}
    Let $\varphi\in\mathcal{C}^d(\Omega;\mathbb{R}^m)$, $d\in\mathbb{N}_0$, $\xi\in\Omega$, and consider System \eqref{system_z}, where $F\in\mathcal{C}^{\max\{0,d-1\}}(\Omega;\mathbb{R}^n)$ and $\Omega$ is positively invariant with respect to the system. Then, for $t\geq 0$, $k\in\{0,\dots,d\}$, \begin{equation*}L^k_F\varphi(z(t;\xi))=\dfrac{\mathrm{d}^k}{\mathrm{d} t^k}\varphi (z(t;\xi)) \ \text{ and, in particular, } \ L^k_F\varphi(\xi)=\left.\dfrac{\mathrm{d}^k}{\mathrm{d} t^k}\varphi (z(t;\xi))\right|_{t=0}.\end{equation*}
\end{lemma}
\begin{proof}
    For any $\xi\in\Omega$, $t\geq 0$, if $d\geq1$, we know that \eqref{LFt}-\eqref{L1xi} is true. By induction, if $d\geq2$, assume that \begin{equation*}L^{k-1}_F\varphi(z(t;\xi))=\dfrac{\mathrm{d}^{k-1}}{\mathrm{d} t^{k-1}}\varphi (z(t;\xi)) \text{ is true for some $k\in\{2,\dots,d\}$, and hence } L^{k-1}_F\varphi(\xi)=\left.\dfrac{\mathrm{d}^{k-1}}{\mathrm{d} t^{k-1}}\varphi (z(t;\xi))\right|_{t=0}.\end{equation*}Then, for $k$, we have \begin{equation*}\dfrac{\mathrm{d}^k}{\mathrm{d} t^k}\varphi(z(t;\xi))=\dfrac{\mathrm{d}}{\mathrm{d} t}L_F^{k-1}\varphi(z(t;\xi))=\mathrm{D}L_F^{k-1}\varphi(z(t;\xi))F(z(t;\xi))=L_F^k\varphi(z(t;\xi)), \text{ and } L_F^k\varphi(\xi)=\left.\dfrac{\mathrm{d}^k}{\mathrm{d} t^k}\varphi (z(t;\xi))\right|_{t=0},\end{equation*}as we wanted to prove.
\end{proof}

Let us now, given $\theta\in\Theta$, denote $h_{\theta}( \cdot )=h(\cdot ,\theta)$ and $f_{\theta}( \cdot )=f(\cdot  ,\theta)$, and assume $h_{\theta}\in\mathcal{C}^{d}(\Omega;\mathbb{R}^m), \ f_{\theta}\in\mathcal{C}^{\max\{0,d-1\}}(\Omega;\mathbb{R}^n)$, for some $d\in\mathbb{N}\cup\{0\}$. Then, $y_{(\xi,\theta)}\in\mathcal{C}^{d}(\mathcal{I};\mathbb{R}^m)$. Indeed, for all $t\geq 0$, \begin{equation*}\dot{y}_{(\xi,\theta)}(t)=\dfrac{\mathrm{d}}{\mathrm{d} t}h(x(t;\xi,\theta))=\mathrm{D} h_{\theta}(x(t;\xi,\theta)) f_{\theta}(x(t;\xi,\theta)) \implies \dot{y}_{(\xi,\theta)}(t)=L^1_{f_{\theta}}h_{\theta}(x(t;\xi,\theta)).\end{equation*}
By Lemma \ref{Lxi}, we have, for all $t\geq 0$ and $k\in\{1,\dots,d\}$, \begin{equation*}y^{(k)}_{(\xi,\theta)}(t)=\dfrac{\mathrm{d}^k}{\mathrm{d} t^k}h(x(t;\xi,\theta))=L_{f_{\theta}}^{k}h_{\theta}(x(t;\xi,\theta)) \text{ and, hence, } y^{(k)}_{(\xi,\theta)}(0)=\left.\dfrac{\mathrm{d}^k}{\mathrm{d} t^k}h(x(t;\xi,\theta))\right|_{t=0}=L_{f_{\theta}}^{k}h_{\theta}(\xi),\end{equation*}where we denote $y^{(k)}=\dfrac{\mathrm{d}^k y}{\mathrm{d} t^k}$, $k\in \mathbb{N},$ when $y\in\mathcal{C}^k(\mathcal{I})$. Then, denoting $y^{(0)}=y$, define $\mathcal{L}_{f_{\theta},h_{\theta},d}:\Omega\rightarrow\RR^{m\times(d+1)}$ as
\begin{equation*}\mathcal{L}_{f_{\theta},h_{\theta},d}(\xi)=\left(y^{(0)}_{(\xi,\theta)}(0),\dots,y^{(d)}_{(\xi,\theta)}(0)\right).\end{equation*}

Roughly speaking, the Lie derivatives of an output of the considered system allow to express the time derivatives of the output not as functions of time, but as functions of the state vector of the system.

In the following, we denote $y_{(\xi,\theta)}$, $\dot{y}_{(\xi,\theta)}$ and $y^{(k)}_{(\xi,\theta)}$ as $y$, $\dot{y}$ and $y^{(k)}$, respectively, when the context is clear.

Let us now present our main results along with the algorithms to recover both the parameters and the initial condition.

\subsection{Main results}

We revisit and extend some results from \cite{Kubik} to establish sufficient hypotheses for System \eqref{general_sys} to be observable, identifiable or jointly observable-identifiable. These results generalize classical approaches (see \cite{Kubik} for a thorough review).

We start recalling a classical result on observability based on Lie derivatives (\cite{Cunniffeetal, HighGains, ORC, Luenberger}). This result is of particular interest in the nonlinear context (for which the Cayley-Hamilton theorem is not available, see \cite[Section 2.2]{Cunniffeetal}). We present the result and a short proof adapted to our framework.

\begin{theorem}
\label{generalization_obs}
Let $\Omega\subset\mathbb{R}^n$ positively invariant with respect to the ODE system given in \eqref{general_sys}, $h_{\theta,i}\in\mathcal{C}^{d_i}(\Omega;\mathbb{R}^m)$, for some $d_i\in\mathbb{N}\cup\{0\}$, $i\in\{1,\dots,m\}$, with $m$ the number of scalar outputs, and $f_{\theta}\in\mathcal{C}^{d-1}(\Omega;\mathbb{R}^n)$, $d=\max\{1,d_1,\dots,d_m\}$, for any $\theta\in\Theta$. If \begin{equation*}\begin{array}{lclcl}
\mathcal{L}_{f_{\theta},h_{\theta},\{d_1,\dots,d_m\}} & : & \Omega & \rightarrow & \RR^{m+d_1+\dots+d_m} \\ && \xi & \mapsto & \left(\mathcal{L}_{f_{\theta},h_{\theta,1},d_1}(\xi), \dots, \mathcal{L}_{f_{\theta},h_{\theta,m},d_m}(\xi) \right)
\end{array}
\end{equation*}
is injective in $\Omega$, then System \eqref{general_sys} is observable on $\Omega$ in any semi-open interval $[a,b)\subset\mathcal{I}$ with parameters in $\Theta$.
\end{theorem}
\begin{proof}
    Given $\theta\in\Theta$, let $x_0,x_0'\in\Omega$ such that \begin{equation*}h(x(t;x_0,\theta),\theta)=h(x(t;x_0',\theta),\theta), \quad \forall\, t\in[a,b).\end{equation*}Given $\tilde{t}\in[a,b)$, let $\xi=x(\tilde{t};x_0,\theta)$ and $\xi'=x(\tilde{t};x_0',\theta).$ Since $\Omega$ is positively invariant w.r.t. the ODE system given in \eqref{general_sys}, $\xi$, $\xi'\in\Omega$. Then, \begin{equation*}h(x(t;\xi,\theta),\theta)=h(x(t;\xi',\theta),\theta), \quad \forall\, t\in[0,b-\tilde{t}).\end{equation*}Notice that $0\in[a-\tilde{t},b-\tilde{t})$ and $[0,b-\tilde{t})\subset\mathcal{I}$ is non-empty. This implies that \begin{equation*}y^{(k)}_{(\xi,\theta),i}(t)=y^{(k)}_{(\xi',\theta),i}(t),\quad \forall\, t\in[0,b-\tilde{t}), \ k\in\{0,\dots,d_i\},\ i\in\{1,\dots,m\},\end{equation*}where $y^{(k)}_{(\xi,\theta)}=\left(y^{(k)}_{(\xi,\theta),1},\dots,y^{(k)}_{(\xi,\theta),m}\right)$. In particular, $y_{(\xi,\theta),i}^{(k)}(0)=y_{(\xi',\theta),i}^{(k)}(0)$, $k\in\{0,\dots,d_i\}$, $i\in\{1,\dots,m\}$. Then, \begin{equation*}\mathcal{L}_{f_{\theta},h_{\theta},\{d_1,\dots,d_m\}}(\xi)=\mathcal{L}_{f_{\theta},h_{\theta},\{d_1,\dots,d_m\}}(\xi').\end{equation*}Since $\mathcal{L}_{f_{\theta},h_{\theta},\{d_1,\dots,d_m\}}$ is injective in $\Omega$, for any $\theta\in\Theta$, and $\xi,\xi'\in\Omega$ due to the positive invariance, this implies that $\xi=\xi'$, i.e., $x(\tilde{t};x_0,\theta)=x(\tilde{t};x_0',\theta)$. Due to the uniqueness of solutions of System \eqref{general_sys}, this implies that $x_0=x_0'$.
    
    Hence, System \eqref{general_sys} is observable on $\Omega$ in any $[a,b)\subset\mathcal{I}$ with parameters in $\Theta$.
\end{proof}

For {$(\xi,\theta)\in\Gamma_{\Theta}$}, $\{d_i\}_{i=1}^m\subset\NN_0$, we define the notation $$ \mathbf{y}_{(\xi,\theta)}^{(d_1,\dots,d_m)}(t)\coloneqq \left(y_{(\xi,\theta),1}^{(0)}(t), \dots,y_{(\xi,\theta),1}^{(d_1)}(t),\dots,y_{(\xi,\theta),m}^{(0)}(t), \dots, y^{(d_m)}_{(\xi,\theta),m}(t)\right),$$
and may use the shorthand notation $\mathbf{y}^{(d_1,\dots,d_m)}$ when there is no ambiguity.

The following is the main result in \cite{Kubik}, adapted to Definition \ref{identifiability_set}, which provides a sufficient condition to ensure recoverability of parameters.

\begin{theorem}
\label{generalization_ident}
   Let $h_{\theta,i}\in\mathcal{C}^{d_i'}(\Omega;\RR)$, for some $d_i'\in\NN_0$\footnote{We use a different notation with respect to Theorem \ref{generalization_obs} to explicitly state that these orders may be different.}, $i\in\{1,\dots,m\}${, and} $f_{\theta}\in\mathcal{C}^{d'-1}(\Omega;\RR^n)$, $d'=\max\{1,d_1',\dots,d_m'\}$, for any $\theta\in\Theta$. Consider $\mathcal{D}\subset\RR^{d_1'+\dots+d_m'+m}$ such {that, for all} {$(\xi,\theta)\in\Gamma_{\Theta}$}, ${\mathbf{y}^{(d_1',\dots,d_m')}(t)\in\mathcal{D}}$, for all $t\in\mathcal{I}$. {Let $\mathcal{S}\subset\mathcal{I}$ be such that every connected component of $S$ contains an open interval. Assume} there {exist} $g:\mathcal{D}\rightarrow \RR^{q+p}$ and $r:\Theta\rightarrow \RR^q$, for some $q,p\in\NN$, satisfying:
    \begin{enumerate}
        \item[\rm(C1)] $g=(g_{1,0},\dots,g_{1,q_1},\dots,g_{p,0},\dots,g_{p,q_p})$ and $r=(r_{1,1},\dots,r_{1,q_1},\dots,r_{p,1},\dots,r_{p,q_p})$, with $q_1+\dots+q_{p}=q$, satisfy that \begin{equation}\label{linear_eq}
        {g_{j,0}(\mathbf{y}^{(d_1',\dots,d_m')})=\sum_{l=1}^{q_j} r_{j,l}(\theta) g_{j,l}(\mathbf{y}^{(d_1',\dots,d_m')})},\end{equation}
        in $\mathcal{S}$, for all $j\in\{1,\dots,p\}$, for any {$(\xi,\theta)\in\Gamma_{\Theta}$},
        \item[\rm(C2)] $r$ is injective, and
        \item[\rm(C3)] for any $j\in\{1,\dots,p\}$, {$(\xi,\theta)\in\Gamma_{\Theta}$}, we have that {$g_{j,l}(\mathbf{y}^{(d_1',\dots,d_m')}(t))$}, $l\in\{1,\dots,q_j\}$, are linearly independent functions with respect to $t\in\mathcal{S}$.
    \end{enumerate}Then, System \eqref{general_sys} is identifiable on $\Theta$ in $\mathcal{S}$ with initial conditions {consistent with the family $\{\Omega_{\theta}\}_{\theta\in\Theta}$}.
\end{theorem}
\begin{proof}
Given $\theta_0\in\Theta$, $\xi\in\Omega_{\theta_0}$, let $\theta_0'\in\Theta$ such that
\begin{equation*}
h_{\theta_0}(x(t;\xi,\theta_0))=h_{\theta_0'}(x(t;\xi,\theta_0')), \quad \forall\, t\in\mathcal{S},
\end{equation*}
i.e.
\begin{equation*}
y_{(\xi,\theta_0)}(t)=y_{(\xi,\theta_0')}(t),\quad \forall\,t\in\mathcal{S}.
\end{equation*}
Then, since every connected part of $\mathcal{S}$ contains some open interval, this implies that 
\begin{equation*}
y^{(k)}_{(\xi,\theta_0),i}(t)=y^{(k)}_{(\xi,\theta_0'),i}(t),\quad \forall\, t\in\mathcal{S},\, k\in\{0,\dots,d_i'\},\, i\in\{1,\dots,m\}.
\end{equation*}
Since
$$g_{j,0}(\mathbf{y}^{(d_1',\dots,d_m')}_{(\xi,\theta_0)})-g_{j,0}(\mathbf{y}^{(d_1',\dots,d_m')}_{(\xi,\theta_0')})\equiv 0$$
in $\mathcal{S}$, we obtain, from \eqref{linear_eq},
$$\sum_{l=1}^{q_j}\big(r_{j,l}(\theta_0)-r_{j,l}(\theta_0')\big)g_{j,l}(\mathbf{y}^{(d_1',\dots,d_m')}_{(\xi,\theta_0)})\equiv 0,$$in $\mathcal{S},$ for every $j\in\{1,\dots,p\}$. Given the linear independence of $g_{j,l}(\mathbf{y}^{(d_1',\dots,d_m')})$, $l\in\{1,\dots,q_j\}$, in $\mathcal{S}\subset\mathcal{I}$, for every $j\in\{1,\dots,p\}$, {$(\xi,\theta)\in\Gamma_{\Theta}$}, then $r(\theta_0)=r(\theta_0')$. Since $r$ is {injective,} {we have} $\theta_0=\theta_0'$. Hence, System \eqref{general_sys} is identifiable on $\Theta$ in $\mathcal{S}\subset\mathcal{I}$ with initial conditions {consistent with the family $\{\Omega_{\theta}\}_{\theta\in\Theta}$}.
\end{proof}

\begin{remark}\label{sufficient_conditions}
    It is important to remark that Theorem \ref{generalization_ident} gives only sufficient conditions for identifiability, but in Section \ref{particularSIR} we will see an example of a system which is identifiable and does not satisfy these conditions. Nevertheless, as we will see in Algorithm \ref{procedure1}, the conditions presented in Theorems \ref{generalization_obs} and \ref{generalization_ident} will allow us to have joined observability-identifiability of System \eqref{general_sys}. In \cite[Proposition 2.1]{VidalBlanchard2001}, the authors propose a result similar to Theorem \ref{generalization_ident}, giving conditions for it to be also necessary, in a context of rational equations and under the availability of observations at initial time.
\end{remark}

Therefore, given a system of ODEs, along with some (partial) observations, we can check the hypotheses in Theorems \ref{generalization_obs} and \ref{generalization_ident} in order to determine the observability, identifiability or joined observability-identifiability of our model and recover the initial condition and parameter vector. Notice that we need to be careful when choosing the positively invariant sets we are going to work with, so that the hypotheses mentioned above are satisfied. {Next, we present a theorem which} shows that, under certain conditions, {given} $y_{(x_0,\theta_0)}$ in some time set, {we can recover} {$(x_0,\theta_0)\in\Gamma_{\Theta}$}.

\begin{theorem}\label{det_univocally}
    For each $\theta\in\Theta$, let $\Omega_{\theta}\subset\Omega$ be positively invariant with respect to the ODE system in \eqref{general_sys}. Let {$(x_0,\theta_0)\in\Gamma_{\Theta}$}. Assume we know $y_{(x_0,\theta_0)}$ in $\mathcal{S}\subset\mathcal{I}$, $\mathcal{S}$ such that every connected component contains an open interval. If the hypotheses of Theorems \ref{generalization_obs} and \ref{generalization_ident} are satisfied, then {System \eqref{general_sys} is jointly observable-identifiable in $\mathcal{S}$ and} we can reconstruct the pair $(x_0,\theta_0)$ univocally using the values of $y_{(x_0,\theta_0)}$ and its derivatives at, at most, $q+1=q_1+\dots+q_p+1$ suitable values of $t\in\mathcal{S}$.
\end{theorem}
\begin{proof} 
Let $\phi_{j,l}(t)=g_{j,l}(\mathbf{y}^{(d_1',\dots,d_m')}_{(x_0,\theta_0)}(t))$, $t\in\mathcal{S}$, $l\in\{0,\dots,q_j\},\ j\in\{1,\dots,p\}${. By} hypothesis, {$\{\phi_{j,l}\}_{l=1}^{q_j}$} are linearly independent in $\mathcal{S}$, for each $j\in\{1,\dots,p\}$. {Then,} for every $j\in\{1,\dots,p\}$, there exist {(see \cite[Lemma 1]{Kubik})} $q_j$ different time instants {$\{t_{j,\ell}\}_{\ell=1}^{q_j}\subset\mathcal{S}$} such that {$$\det((\phi_{j,l}(t_{j,\ell}))_{l,\ell=1,\dots,q_j})\neq 0.$$}
Then, there exists a unique solution $\sigma$ to 
{\begin{equation}\label{linsys}\left(\phi_{j,1}(t_{j,\ell}) \ \dots \  \phi_{j,q_j}(t_{j,\ell})\right)\sigma=
        \phi_{j,0}(t_{j,\ell}),\ \forall \ \ell\in\{1,\dots,q_j\}.\end{equation}}
{Attend}ing to \eqref{linear_eq}, it satisfies $\sigma_{j,l}=r_{j,l}(\theta_0)$, $l\in\{1,\dots,q_j\}, \ j\in\{1,\dots,p\}$. {Since $r$ is injective}, such that $r^{-1}:r(\Theta)\rightarrow \Theta$, we {can recover} our original parameter vector $$\theta_0 =r^{-1}(\sigma_{1,1},\dots,\sigma_{p,q_p}).$$
        
Finally, to recover {$x_0$,} take some time $\tilde{t}\in\mathcal{S}$, which can be some $\tilde{t}\in\{t_{1,1},\dots,t_{p,q_p}\}$. {Since $\mathcal{L}_{f_{\theta_0},h_{\theta_0},\{d_1,\dots,d_m\}}$ is injective} in $\Omega$, there exists a unique $\tilde{\xi}\in\Omega$ such that $$\tilde{\xi}=\mathcal{L}^{-1}_{f_{\theta_0},h_{\theta_0},\{d_1,\dots,d_m\}} {(\mathbf{y}^{(d_1,\dots,d_m)}_{(x_0,\theta_0)}(\tilde{t}))},$$noticing that {$\mathbf{y}^{(d_1,\dots,d_m)}_{(x_0,\theta_0)}(\tilde{t})=\mathbf{y}^{(d_1,\dots,d_m)}_{(\tilde{\xi},\theta_0)}(0)$}. Since it is unique, it must satisfy $\tilde{\xi}\in\Omega_{\theta_0}$. We can recover {$x_0\in\Omega_{\theta_0}$} integrating backwards the ODE system in System \eqref{general_sys} knowing $\tilde{\xi}$, $\tilde{t}$ and $\theta_0$ ({since} $f_{\theta_0}$ is Lipschitz in $\Omega$ and, in particular, in $\Omega_{\theta_0}$ positively invariant w.r.t. the ODE system {of \eqref{general_sys}}). If $0\in\mathcal{S}$, {we directly choose} $\tilde{t}=0$.

This is, we have recovered $\theta_0$ and $x_0$ from the data univocally knowing {$y_{(x_0,\theta_0)}$ in $\mathcal{S}$,} using its value and the value of its derivatives at $q+1$ (at most) different time instants.
\end{proof}

In the following Algorithm \ref{procedure1}, we present a procedure to recover the unknowns given some observations, based on the proof of Theorem \ref{det_univocally}.

\begin{algorithm}\label{procedure1} Assume that we know $y_{(x_0,\theta_0)}$ (satisfying System \eqref{general_sys}) in $\mathcal{S}\subset\mathcal{I}$, such that every connected of $\mathcal{S}$ component contains an open interval, and $x_0\in\Omega$ and/or $\theta_0\in\Theta$ are unknown. The procedure to recover the unknowns is the following:
\begin{itemize}
    \item[Step 1.] If $x_0$ is unknown, find $\Omega_1\subset\Omega$ positively invariant with respect to the ODE system given in \eqref{general_sys} such that $x_0\in\Omega_1$, and $d_1,\dots,d_m\in\mathbb{N}_0$ such that, for any $\theta\in\Theta$, the following function is injective in $\Omega_1$:
    \begin{equation*}
    \mathcal{L}_{\theta,\{d_1,\dots,d_m\}}: \xi\mapsto \mathbf{y}^{(d_1,\dots,d_m)}_{(\xi,\theta)}(0).
    \end{equation*}
    Note: According to Theorem \ref{generalization_obs}, this ensures that System \eqref{general_sys} is observable on $\Omega_1$ in any semi-open interval $[a,b)\subset\mathcal{I}$ with parameters in $\Theta$.
    
    \item[Step 2.] If $\theta_0$ is unknown, find $\Omega_{2,\theta}\subset\Omega$, for any $\theta\in\Theta$, positively invariant with respect to the ODE system given in \eqref{general_sys} such that $x_0\in\Omega_{2,\theta_0}$, and maps $g: \mathcal{D}\rightarrow\mathbb{R}^{q+p}$ and $r: \Theta\rightarrow \mathbb{R}^p$, for some suitable $\mathcal{D}\subset \mathbb{R}^{d_1'+\dots+d_m'+m}$, $d_1',\dots,d_m'\in\mathbb{N}_0$, $q,p\in\mathbb{N}$, such that (C1), (C2) and (C3) of Theorem \ref{generalization_ident} are satisfied.
    
    Note: According to Theorem \ref{generalization_ident}, this ensures that System \eqref{general_sys} is identifiable on $\Theta$ in $\mathcal{S}$ with initial conditions {consistent with the family $\{\Omega_{2,\theta}\}_{\theta\in\Theta}$}.
    
    \item[Step 3.] If we have defined $\Omega_1$ and $\Omega_{2,\theta}$, set $\Omega_{\theta}=\Omega_1\cap\Omega_{2,\theta}$, for all $\theta\in\Theta$. If we have only defined $\Omega_1$, redefine $\Omega=\Omega_1$. If we have only defined $\Omega_{2,\theta}$, set $\Omega_{\theta}=\Omega_{2,\theta}$, for every $\theta\in\Theta$.
    
    \item[Step 4.] If $\theta_0$ is unknown, for each $j\in\{1,\dots,p\}$, find $q_j$ different time instants $t_{j,1},\dots,t_{j,q_j}\in\mathcal{S}$ such that \begin{equation*}\det\left(\phi_{j,l}(t_{j,\ell})_{l,\ell=1,\dots,q_j}\right)\neq 0,\end{equation*}where $\phi_{j,l}=g_{j,l}\left(\mathbf{y}^{(d_1',\dots,d_m')}_{(x_0,\theta_0)}\right)$.
    
    Note: These time instants exist due to condition (C3) of Theorem \ref{generalization_ident}, which is verified in Step 2, and \cite[Lemma 1]{Kubik}.
    
    \item[Step 5.] If $\theta_0$ is unknown, for each $j\in\{1,\dots,p\}$, solve the following linear system which, according to Step 4, has a unique solution $\sigma_j=(\sigma_{j,1},\dots,\sigma_{j,q_j})$: \begin{equation}\label{step4proc1}\left(\begin{array}{ccc}
        \phi_{j,1}(t_{j,1}) & \cdots & \phi_{j,q_j}(t_{j,1}) \\
        \vdots & \ddots & \vdots \\
        \phi_{j,1}(t_{j,q_j}) & \cdots & \phi_{j,q_j}(t_{j,q_j})
    \end{array}\right)\left(\begin{array}{c}
        \sigma_{j,1} \\
         \vdots \\
         \sigma_{j,q_j}
    \end{array}\right)=\left(\begin{array}{c}
        \phi_{j,0}(t_{j,1}) \\
        \vdots \\
        \phi_{j,0}(t_{j,q_j})
        \end{array}\right),\end{equation}where $\phi_{j,0}=g_{j,0}\left(\mathbf{y}^{(d_1',\dots,d_m')}_{(x_0,\theta_0)}\right)$.
        
    \item[Step 6.] If $\theta_0$ is unknown, recover $\theta_0$ solving the equation
    \begin{equation*}
    r(\theta_0)=(\sigma_1,\dots,\sigma_p).
    \end{equation*}
    Note: We are taking into account that $r$ is injective and $r(\theta_0)$ is solution of System \eqref{step4proc1}, according to conditions (C1) and (C2) of Theorem \ref{generalization_ident}, which are verified in Step 2.
    
    \item[Step 7.] If $x_0$ is unknown, choose some $\tilde{t}\in\mathcal{S}$ (which can be some $\tilde{t}\in\{t_{1,1},\dots,t_{p,q_p}\}$ if we performed Step 4) and solve for $\tilde{\xi}$ the equation
    \begin{equation*}
    \mathcal{L}_{\theta_0,\{d_1,\dots,d_m\}}(\tilde{\xi})=\mathbf{y}^{(d_1',\dots,d_m')}_{(x_0,\theta_0)}(\tilde{t}).
    \end{equation*}
    Note: We are taking into account that $\mathcal{L}_{\theta,\{d_1,\dots,d_m\}}$ is injective in $\Omega$ and, in particular, in $\Omega_{\theta}$, for any $\theta\in\Theta$, according to Step 1, and $y_{(x_0,\theta_0),i}^{(k)}(\tilde{t})=y_{(\tilde{\xi},\theta_0),i}^{(k)}(0)$, for all $k\in\{0,\dots,d_i\}$, $i\in\{1,\dots,m\}$.
    
    \item[Step 8.] If $x_0$ is unknown and $\tilde{t}\neq 0$, integrate backwards System \eqref{general_sys} from $\tilde{t}$ to $0$ with initial condition $\tilde{\xi}$, knowing $\theta_0$, in order to recover $x_0$.
\end{itemize}
\end{algorithm}

Recall that the time instants required in Step 4 may be anywhere in $\mathcal{S}\subset\mathcal{I}$, and hence may be difficult to find in practice. Nevertheless, in Lemma \eqref{infinitesimal_ident} we give sufficient hypotheses such that System \eqref{general_sys} is identifiable in any semi-open interval $[a,b)\subset\mathcal{S}$; this will imply that we will be able to choose this set of time instants in any of these semi-open intervals. Then, an analogous procedure to this one will be presented in Algorithm \ref{procedure2}.

\begin{remark}\label{range_qi_times}
Notice that, in order to be able to recover $(x_0,\theta_0)$ following the procedure in Step 4 of Algorithm \ref{procedure1}, for each set $\{\phi_{j,1},\dots,\phi_{j,q_j}\}$, $j\in\{1,\dots,p\}$, we need to find $q_j$ different suitable time instants. However, some of these time instants may coincide among different sets of linearly independent functions. Thus, the number of different time instants we need to find is between $\tilde{q}=\max\{q_1,\dots,q_p\}$ and $q=q_1+\dots+q_p$, along maybe with $\tilde{t}=0$, which we can use to recover the initial condition if $0\in\mathcal{S}$ and could be one of the other time instants.

\noindent Recall, moreover, that we do not necessarily need $y_{(x_0,\theta_0)}(t)$, for all $t\in\mathcal{S}$, but it would be sufficient having the values of $y_{(x_0,\theta_0)}$ and its derivatives at the aforementioned different time instants, where the order of the derivatives that we need are the same as for Algorithm \ref{procedure1}.
\end{remark}

\noindent
\textbf{Note:}
    Although this idea of choosing a number of suitable time instants has already been commented in \cite[Remark 3]{Saccomani2003} and \cite{Komatsu2020}, no proof nor more detail is given.
    
\begin{lemma}
\label{infinitesimal_ident}
   Assume the hypotheses of Theorem \ref{generalization_ident} are satisfied for some $\mathcal{S}\subset\mathcal{I}$ and, for any $j\in\{1,\dots,p\}$, $l\in\{1,\dots,q_j\}$ and {$(\xi,\theta)\in\Gamma_{\Theta}$}, the functions $g_{j,l}(\mathbf{y}^{(d_1',\dots,d_m')})$ are analytic on $\mathcal{I}$. Then, System \eqref{general_sys} is identifiable on $\Theta$ in any semi-open interval $[a,b)\subset\mathcal{S}$ with initial conditions {consistent with the family $\{\Omega_{\theta}\}_{\theta\in\Theta}$}{. If $\mathcal{S}$} is connected, it is sufficient for $g_{j,l}(\mathbf{y}^{(d_1',\dots,d_m')})$, $j\in\{1,\dots,p\}$, $l\in\{1,\dots,q_j\}$, to be analytic on $\mathcal{S}$.
\end{lemma}

\begin{proof}
    Given $\theta_0\in\Theta$, $\xi\in\Omega_{\theta_0}$, let $\theta_0'\in\Theta$ such that
    $$h_{\theta_0}(x(t;\xi,\theta_0))=h_{\theta_0'}(x(t;\xi,\theta_0')),\quad \forall\, t\in[a,b),$$
    i.e.,
    $$y_{(\xi,\theta_0)}(t)=y_{(\xi,\theta_0')}(t),
    \quad \forall\, t\in[a,b).$$
    {Then,}  {$$y^{(k)}_{(\xi,\theta_0),i}(t)=y^{(k)}_{(\xi,\theta_0'),i}(t),\quad \forall\, t\in[a,b),$$} $k\in\{0,\dots,d_i'\},\ i\in\{1,\dots,m\}.$ {Since}, for any $j\in\{1,\dots,p\}$, $l\in\{1,\dots,q_j\}$ and {$(\xi,\theta)\in\Gamma_{\Theta}$}, the functions $g_{j,l}(\mathbf{y}^{(d_1',\dots,d_m')})$ are analytic {in $\mathcal{I}$}, then the function {$$ G_{j,\theta}(\mathbf{y}^{(d_1',\dots,d_m')})\coloneqq\sum_{l=1}^{q_j}r_{j,l}(\theta)g_{j,l}(\mathbf{y}^{(d_1',\dots,d_m')})$$}
    is also analytic in $\mathcal{I}$. Given that $[a,b)\subset\mathcal{S}$, {$$G_{j,\theta}(\mathbf{y}^{(d_1',\dots,d_m')})=g_{j,0}(\mathbf{y}^{(d_1',\dots,d_m')}),$$ in $[a,b)$, for any $\theta\in\Theta$. {Since} $$g_{j,0}(\mathbf{y}^{(d_1',\dots,d_m')}_{(\xi,\theta_0)})=g_{j,0}(\mathbf{y}^{(d_1',\dots,d_m')}_{(\xi,\theta_0')}),$$ in $[a,b)${, we have} $$G_{j,\theta_0}(\mathbf{y}^{(d_1',\dots,d_m')}_{(\xi,\theta_0)})=G_{j,\theta_0'}(\mathbf{y}^{(d_1',\dots,d_m')}_{(\xi,\theta_0')}),$$
    in $[a,b)$.} This implies that, for every $j\in\{1,\dots,p\}$, $$R_j\coloneqq \sum_{l=1}^{q_j}\big(r_{j,l}(\theta_0)-r_{j,l}(\theta_0')\big)g_{j,l}(\mathbf{y}^{(d_1',\dots,d_m')}_{(\xi,\theta_0)})\equiv 0$$
    in $[a,b)$. Due to the analyticity of $g_{j,l}(\mathbf{y}^{(d_1',\dots,d_m')})$, $l\in\{1,\dots,q_j\}$, in $\mathcal{I}$, $R_j$ is also analytic in $\mathcal{I}${. If} $R_j\equiv 0$ in $[a,b)$, {then} $R_j\equiv 0$ in $\mathcal{I}$ (\cite[Theorem 8.5]{Rudin}). Therefore, since $g_{j,l}(\mathbf{y}^{(d_1',\dots,d_m')})$, $l\in\{1,\dots,q_j\}$, are linearly independent in $\mathcal{S}\subset\mathcal{I}$, they are {also} linearly independent in $\mathcal{I}$, and {hence, for} $R_j$ to be 0 in $\mathcal{I}$, for all $j\in\{1,\dots,p\}$, we need $r(\theta_0)=r(\theta_0')$. Since $r$ is {injective}, this {implies} $\theta_0=\theta_0'$.

    {If $\mathcal{S}$} is connected, $g_{j,l}(\mathbf{y}^{(d_1',\dots,d_m')})$, $l\in\{1,\dots,q_j\}$, analytic in $\mathcal{S}$ implies $R_j$ {is analytic} in $\mathcal{S}$, and $R_j\equiv 0$ in $[a,b)\subset\mathcal{S}$ {implies} $R_j\equiv 0$ in $\mathcal{S}$ (if $\mathcal{S}$ is not connected, we can only assure $R_j\equiv 0$ in the connected {component containing} $[a,b)$). {Thus}, we conclude analogously using the linear independence of $g_{j,l}(\mathbf{y}^{(d_1',\dots,d_m')})$, $l\in\{1,\dots,q_j\}$, in $\mathcal{S}$. {Hence,} System \eqref{general_sys} is identifiable on $\Theta$ in any $[a,b)\subset\mathcal{S}$ with initial conditions {consistent with the family $\{\Omega_{\theta}\}_{\theta\in\Theta}$}.
\end{proof}

Taking into account Theorem \ref{generalization_obs} and Lemma \ref{infinitesimal_ident}, we {present a lemma} that shows that we {may} recover $x_0$ and $\theta_0$ knowing $y_{(x_0,\theta_0)}$ only in some $[a,b)\subset\mathcal{I}$.

\begin{lemma}\label{recover_infi_lindep}
    For each $\theta\in\Theta$, let $\Omega_{\theta}\subset\Omega$ be positively invariant with respect to the ODE system in \eqref{general_sys}. Let {$(x_0,\theta_0)\in\Gamma_\Theta$}. Assume we know $y_{(x_0,\theta_0)}$ in some $[a,b)\subset\mathcal{I}$. If the hypotheses of Theorem \ref{generalization_obs} and Lemma \ref{infinitesimal_ident} are satisfied for some $\mathcal{S}\subset\mathcal{I}$ such that $[a,b)\subset\mathcal{S}$, then {System \eqref{general_sys} is jointly observable-identifiable in $[a,b)$, and} we can reconstruct $(x_0,\theta_0)$ univocally using the values of $y_{(x_0,\theta_0)}$ and its derivatives at a finite amount of suitable values {in $[a,b)$}. In particular, the required number of time points is between $\tilde{q}=\max\{q_1,\dots,q_p\}$ and $q=q_1+\dots+q_p$.
\end{lemma}

\begin{proof}
    {We recover $\theta_0$ analogously to} the proof of Theorem \ref{det_univocally}. {We only} need to see that, given $j\in\{1,\dots,p\}$, since the functions $g_{j,l}(\mathbf{y}^{(d_1',\dots,d_m')})$, $l\in\{1,\dots,q_j\}$, are linearly independent in some $\mathcal{S}\subset\mathcal{I}$ and analytic in $\mathcal{I}$, they are linearly independent in {$[a,b)\subset\mathcal{S}$. Given} $j\in\{1,\dots,p\}$, assume that $g_{j,l}(\mathbf{y}^{(d_1',\dots,d_m')})$, $l\in\{1,\dots,q_j\}$, are linearly dependent in $[a,b)$, i.e., there exist {$\{a_{j,l}\}_{l=1}^{q_j}\subset\RR$} not all of them null such that,
    $$G_j(t)=\sum_{l=1}^{q_j}a_{j,l}g_{j,l}(\mathbf{y}^{(d_1',\dots,d_m')}(t))=0,\ \forall \ t\in[a,b).$$Since $g_{j,l}(\mathbf{y}^{(d_1',\dots,d_m')})$, $l\in\{1,\dots,q_j\}$, are analytic in $\mathcal{I}$, then so it is $G_j$. This implies that {(see \cite[Theorem 8.5]{Rudin})} $G_j\equiv 0$ in $\mathcal{I}$ and, thus, $g_{j,l}(\mathbf{y}^{(d_1',\dots,d_m')})$, $l\in\{1,\dots,q_j\}$, are linearly dependent in $\mathcal{I}$, and hence in $\mathcal{S}$, which is a contradiction.
    
    Moreover, if $\mathcal{S}$ is connected, it is enough asking for $g_{j,l}(\mathbf{y}^{(d_1',\dots,d_m')})$, $l\in\{1:q_j\}$, analytic in $\mathcal{S}$, since, then, so it is $G_j$. {Therefore}, $G_j\equiv0$ in $[a,b)\subset\mathcal{S}$ connected implies {(see \cite[Theorem 8.5]{Rudin}) that} $G_j\equiv 0$ in $\mathcal{S}$, which leads to the same contra{diction. There}fore, for each $j\in\{1,\dots,p\}$, the functions {$\{\phi_{j,l}\}_{l=1}^{q_j}$}, with $\phi_{j,l}=g_{j,l}(\mathbf{y}^{(d_1',\dots,d_m')}_{(x_0,\theta_0)})$, $l\in\{1,\dots,q_j\}$, are linearly independent in $[a,b)$ and we can conclude analogously {to Theorem} \ref{det_univocally}, along with Remark \ref{range_qi_times}, {choosing} $t_{1,1},$ $\dots,$ $t_{p,q_p}\in[a,b)$.
    
    {On the} other hand, to recover the initial condition, we proceed as in the proof of Theorem \ref{det_univocally}.
    
    {Hence,} we are able to recover $(x_0,\theta_0)$ univocally when knowing $y_{(x_0,\theta_0)}$ in $[a,b)$, using its values and the values of its derivatives at some finite set of time instants in $[a,b)$; concretely, between $\tilde{q}$ and $q$ d{distinct} suitable {time points.}
\end{proof}

We present in Algorithm \ref{procedure2} a method to recover $x_0$ and/or $\theta_0$ knowing $y_{(x_0,\theta_0)}$ only in some $[a,b)\subset\mathcal{I}$, based on the proof of previous Lemma \ref{recover_infi_lindep}. In particular, we will need to find between $\tilde{q}=\max\{q_1,\dots,q_p\}$ and $q=q_1+\dots+q_p$ different time instants in $[a,b)$ when some analyticity properties are satisfied.

\begin{algorithm}\label{procedure2} In a way similar to that of Algorithm \ref{procedure1}, we describe here a slightly different constructive method to recover $x_0\in\Omega$ and/or $\theta_0\in\Theta$ assuming that we know $y_{(x_0,\theta_0)}$ (satisfying System \eqref{general_sys}) in some $[a,b)\subset\mathcal{I}$ and some analyticity hypotheses are fulfilled. Since the procedure is similar to the one in Algorithm \ref{procedure1}, we only state here the steps that change:
\begin{itemize}
    \item[Step 2.] If $\theta_0$ is unknown, find $\Omega_{2,\theta}\subset\Omega$, for any $\theta\in\Theta$, positively invariant with respect to the ODE system given in \eqref{general_sys} such that $x_0\in\Omega_{2,\theta_0}$, and maps $g: \mathcal{D}\rightarrow\mathbb{R}^{q+p}$ and $r: \Theta\rightarrow \mathbb{R}^p$, for some suitable $\mathcal{D}\subset \mathbb{R}^{d_1'+\dots+d_m'+m}$, $d_1',\dots,d_m'\in\mathbb{N}\cup\{0\}$, $q,p\in\mathbb{N}$, such that (C1), (C2) and (C3) of Theorem \ref{generalization_ident} are satisfied, being $\mathcal{S}\subset\mathcal{I}$ in (C2) and (C3) a connected set satisfying $[a,b)\subset\mathcal{S}$, and such that functions $g_{j,l}(\mathbf{y}^{(d_1',\dots,d_m')})$ are analytic in $\mathcal{S}$, for any $j\in\{1,\dots,p\}$, $l\in\{1,\dots,q_j\}$, {$(\xi,\theta)\in\Gamma_{2,\Theta}=\bigcup_{\theta\in\Theta}\Omega_{2,\theta}\times\{\theta\}$}.
    
    Note: According to Lemma \ref{infinitesimal_ident}, this ensures that System \eqref{general_sys} is identifiable on $\Theta$ in any semi-open interval $[a,b)\subset\mathcal{S}$ with initial conditions {consistent with the family $\{\Omega_{2,\theta}\}_{\theta\in\Theta}$}.
    
    \item[Step 4.] Same as in Algorithm \ref{procedure1} but such that, for each $j\in\{1,\dots,p\}$, $t_{j,1},\dots,t_{j,q_j}\in[a,b)$.
    
    Note: These time instants exist due to condition (C3) of Theorem \ref{generalization_ident} (which is verified in Step 2), the analyticity condition required also in Step 2, and \cite[Lemma 1]{Kubik}.
    
    \item[Step 7.] Same as in Algorithm \ref{procedure1}, but such that $\tilde{t}\in[a,b)$.
\end{itemize}
\end{algorithm}

Notice again that we need to find in $[a,b)$ the sets of time instants required in Step 4, and hence again may be difficult to find in practice. However, we can reduce this quantity of time instants to $1$ time if we use higher derivatives of $y$, as will be shown next.

\begin{remark}\label{ydiscret2}
    Recall that, as in Remark \ref{range_qi_times}, we may not need to know $y_{(x_0,\theta_0)}$ in some interval $[a,b)$, but only the values of this function and its derivatives at the time instants indicated in Step 4 and Step 7 in Algorithm \ref{procedure2}, where the needed order of the derivatives is given by the same algorithm.
\end{remark}

\subsection{Using higher-order derivatives}

In this section, we present a methodology which uses higher derivatives of $y$ and may be more convenient in some cases. Some classical differential algebra methodologies use higher-order derivatives and, once they obtain the expressions of these derivatives, they study the identifiability of the system from a differential algebra viewpoint (recall, for example, \cite{Gevers2016, Ljung1994, Ovchinnikov2021}). In the methodology presented in Lemma \ref{recover_infi_wron} and Algorithm \ref{procedure3} we still determine the identifiability of System \eqref{general_sys} studying the linear independence, which in several cases will be easier than what is performed in the other methodologies. The use we make of higher-order derivatives allows to reduce to one the number of time instants that we need to find in $[a,b)\subset\mathcal{I}$ to recover the unknown parameters, once we know the system is identifiable. Therefore, notice that using these higher-order derivatives is an alternative to recover the unknowns using less time instants but is not necessary for proving the identifiability of the system, as it is in the aforementioned works.

\begin{lemma}\label{recover_infi_wron}
   For each $\theta\in\Theta$, let $\Omega_{\theta}\subset\Omega$ be positively invariant with respect to the ODE system in \eqref{general_sys}. Let {$(x_0,\theta_0)\in\Gamma_{\Theta}$}. Assume we know $y_{(x_0,\theta_0)}$ in some semi-open interval $[a,b)\subset\mathcal{I}$, and the hypotheses of Theorem \ref{generalization_obs} and Lemma \ref{infinitesimal_ident} are satisfied for some $\mathcal{S}\subset\mathcal{I}$ such that $[a,b)\subset\mathcal{S}$. If we further assume that $h_{\theta,i}\in\mathcal{C}^{d_i'+\tilde{q}-1}(\Omega;\mathbb{R})$, $i\in\{1,\dots,m\}$, $f_{\theta}\in\mathcal{C}^{d'+\tilde{q}-2}(\Omega;\mathbb{R}^n)$, for any $\theta\in\Theta$, being $\tilde{q}=\max\{q_1,\dots,q_p\}$, then System \eqref{general_sys} is jointly observable-identifiable and we can reconstruct $(x_0,\theta_0)$ univocally using the value of $y_{(x_0,\theta_0)}$ and its derivatives at only one time in $[a,b)$; in particular, this time can be almost any $t\in[a,b)$.
\end{lemma}

\begin{proof}
    Let us take $\phi_{j,l}=g_{j,l}(\mathbf{y}^{(d_1',\dots,d_m')}_{(x_0,\theta_0)})$ in $[a,b)$, for all $l\in\{0,\dots,q_j\},\ j\in\{1,\dots,p\}$. Since \begin{equation*}
    \phi_{j,0}(t)=\sum_{l=1}^{q_j} r_{j,l}(\theta_0) \phi_{j,l}(t),\quad \forall\, t\in[a,b),\end{equation*} where $r(\theta_0)$ does not depend on time, we differentiate each of these equations, for $j\in\{1,\dots,p\}$, $q_j-1$ times such that \begin{equation*}
    \dfrac{\mathrm{d}^k \phi_{j,0}}{\mathrm{d} t^k}(t)=\sum_{l=1}^{q_j} r_{j,l}(\theta_0) \dfrac{\mathrm{d}^k \phi_{j,l}}{\mathrm{d} t^k}(t), \quad \forall \ t\in[a,b),\ k\in\{0,\dots,q_j-1\},\end{equation*}
By the proof of \cite[Lemma 3]{Kubik}, we know that $\phi_{j,1},\dots,\phi_{j,q_j}$ are linearly independent in $[a,b)$, and hence there exists some $t_j\in[a,b)$ (see \cite{Bocher1900}) such that \begin{equation*}
W_{j}(t_j)=\det\left(\left(\dfrac{\mathrm{d}^{k} \phi_{j,l}}{\mathrm{d} t^{k}}(t_j)\right)_{\begin{subarray}{l} l=1,\dots,q_j \\ k=0,\dots,q_j-1\end{subarray}}\right)\neq 0,\end{equation*} where $W_{j}$ is the Wro\'nskian of $\phi_{j,1},\dots,\phi_{j,q_j}$, i.e., there exists a unique solution $\sigma_{j}$ to
    \begin{equation*}
    \left(\begin{array}{ccc}
        \phi_{j,1}(t_j) & \cdots & \phi_{j,q_j}(t_j) \\[.5em]
        \dfrac{\mathrm{d} \phi_{j,1}}{\mathrm{d} t}(t_j) & \cdots & \dfrac{\mathrm{d} \phi_{j,q_j}}{\mathrm{d} t}(t_j) \\[.5em]
        \vdots & \ddots & \vdots \\[.5em]
        \dfrac{\mathrm{d}^{q_j-1} \phi_{j,1}}{\mathrm{d} t^{q_j-1}}(t_j) & \cdots & \dfrac{\mathrm{d}^{q_j-1} \phi_{j,q_j}}{\mathrm{d} t^{q_j-1}}(t_j)
    \end{array}\right)\left(\begin{array}{c}
        \sigma_{j,1} \\
        \sigma_{j,2} \\
        \vdots \\
        \sigma_{j,q_j}
    \end{array}\right)=\left(\begin{array}{c}
        \phi_{j,0}(t_j) \\[.5em]
        \dfrac{\mathrm{d} \phi_{j,0}}{\mathrm{d} t}(t_j) \\[.5em]
        \vdots \\[.5em]
        \dfrac{\mathrm{d}^{q_j-1}\phi_{j,0}}{\mathrm{d} t^{q_j-1}}(t_j)
    \end{array}\right).\end{equation*}

    Moreover, $W_j$ is also analytic, and therefore it can only vanish at isolated points \cite[Theorem 8.5]{Rudin}, i.e., the previous equation is fulfilled for almost every $t_j\in[a,b)$. Furthermore, since each $W_j$, $j\in\{1,\dots,p\}$, only vanishes at a countable set of times, then all of them are non-null simultaneously for almost every $t\in[a,b)$. Hence, we can choose almost any time $\tilde{t}\in[a,b)$ such that the previous system has a unique solution considering $t_j=\tilde{t}$, for every $j\in\{1,\dots,p\}$.
    
    In a way similar to that of the proof of \cite[Theorem 3]{Kubik}, we may recover our original parameter vector $\theta_0$ as
    \begin{equation*}
    \theta_0 =r^{-1}(\sigma_{1,1},\dots,\sigma_{p,q_p}).\end{equation*}
    Having recovered $\theta_0$, we can recover $x_0$ in the same way as exposed in the proof of \cite[Lemma 3]{Kubik}.

    Hence, we are able to recover $(x_0,\theta_0)$ univocally when knowing $y_{(x_0,\theta_0)}(t)$, $t\in[a,b)$, using its value and the value of its derivatives at one time in $[a,b)$, which can be almost any $t\in[a,b)$.
\end{proof}

Therefore, if we ask for higher regularity of the components of $y_{(\xi,\theta)}$, {$(\xi,\theta)\in\Gamma_{\Theta}$}, than required in \cite[Lemma 3]{Kubik}, we can recover $x_0$ and $\theta_0$ choosing only one time in $[a,b)$, instead of a number of distinct time instants between $\tilde{q}=\max\{q_1,\dots,q_p\}\geq 1$ and $q=q_1+\dots+q_p\geq p$ (see Remark \ref{range_qi_times}). We present next the associated algorithm.

\begin{algorithm}\label{procedure3}
In a way similar to those of Algorithm \ref{procedure1} and Algorithm \ref{procedure2}, we describe here a different constructive method to recover $x_0\in\Omega$ and/or $\theta_0\in\Theta$ assuming that we know $y_{(x_0,\theta_0)}$ (satisfying System \eqref{general_sys}) in some $[a,b)\subset\mathcal{I}$, some analyticity properties are fulfilled and $y_{(x_0,\theta_0)}$ is smooth enough. This method is presented in the proof of Lemma \ref{recover_infi_wron}, and it requires to choose only one time in $[a,b)$. However, it needs the use of higher derivatives of $y_{(x_0,\theta_0)}$, which are not necessary in the constructive methods explained in Algorithm \ref{procedure1} and Algorithm \ref{procedure2}. We only state here the steps that are different from those in Algorithms \ref{procedure1} and \ref{procedure2}:
\begin{itemize}
    \item[Step 2.] Same as in Step 2 of Algorithm \ref{procedure2}, along with $y_{(\xi,\theta),i}\in\mathcal{C}^{d_i'+\tilde{q}-1}$, for every $(\xi,\theta)\in\Gamma_{2,\Theta}$ and each $i\in\{1,\dots,m\}$, with $\tilde{q}=\max\{q_1,\dots,q_p\}$.
    
    \item[Step 4.] If $\theta_0$ is unknown, for each $j\in\{1,\dots,p\}$, differentiate with respect to time $q_j-1$ times the equation \begin{equation*}\phi_{j,0}=\sum_{l=1}^{q_j} r_{j,l}(\theta_0) \phi_{j,l},\end{equation*}
    where $\phi_{j,\ell}=g_{j,\ell}\left(\mathbf{y}^{(d_1',\dots,d_m')}_{(x_0,\theta_0)}\right)$, $\ell\in\{0,l\}$.
    
    Note: We can differentiate due to the analyticity condition on functions $g_{j,l}$ and the regularity conditions on $y$ required in Step 2.
    
    \item[Step 5.] If $\theta_0$ is unknown, choose a time $\tilde{t}\in[a,b)$ such that \begin{equation}\label{step5proc3}W_j(\tilde{t})=\det\left(\left(\dfrac{\mathrm{d}^{k} \phi_{j,l}}{\mathrm{d} t^{k}}(\tilde{t})\right)_{\begin{subarray}{l} l=1,\dots,q_j \\ k=0,\dots,q_j-1\end{subarray}}\right)\neq0,\quad \forall\, j\in\{1,\dots,p\},\end{equation}where $W_j$ is the Wro\'nskian of $\{\phi_{j,1},\dots,\phi_{j,q_j}\}$.
    
    Note: This time can be almost any time in $[a,b)$ due to condition (C3) of Theorem \ref{generalization_ident}, which is verified in Step 2, and the analyticity condition required in Step 2.
    
    \item[Step 6.] If $\theta_0$ is unknown, for each $j\in\{1,\dots,p\}$, solve the following linear system which, due to Step 5, has a unique solution $\sigma_j=(\sigma_{j,1},\dots,\sigma_{j,q_j})$: \begin{equation}\label{step6proc3}\left(\begin{array}{ccc}
        \phi_{j,1}(\tilde{t}) & \cdots & \phi_{j,q_j}(\tilde{t}) \\[.5em]
        \dfrac{\mathrm{d} \phi_{j,1}}{\mathrm{d} t}(\tilde{t}) & \cdots & \dfrac{\mathrm{d} \phi_{j,q_j}}{\mathrm{d} t}(\tilde{t}) \\[.5em]
        \vdots & \ddots & \vdots \\[.5em]
        \dfrac{\mathrm{d}^{q_j-1} \phi_{j,1}}{\mathrm{d} t^{q_j-1}}(\tilde{t}) & \cdots & \dfrac{\mathrm{d}^{q_j-1} \phi_{j,q_j}}{\mathrm{d} t^{q_j-1}}(\tilde{t})
    \end{array}\right)\left(\begin{array}{c}
        \sigma_{j,1} \\
        \sigma_{j,2} \\
        \vdots \\
        \sigma_{j,q_j}
    \end{array}\right)=\left(\begin{array}{c}
        \phi_{j,0}(\tilde{t}) \\[.5em]
        \dfrac{\mathrm{d} \phi_{j,0}}{\mathrm{d} t}(\tilde{t}) \\[.5em]
        \vdots \\[.5em]
        \dfrac{\mathrm{d}^{q_j-1}\phi_{j,0}}{\mathrm{d} t^{q_j-1}}(\tilde{t})
    \end{array}\right).\end{equation}
    
    \item[Step 7.] Same as in Step 6 of Algorithm \ref{procedure1}, but considering that $r(\theta_0)$ is solution of System \eqref{step6proc3}.
    
    \item[Step 8.] Same as in Step 7 of Algorithm \ref{procedure2}, but now the time we choose can be the same as $\tilde{t}$ in Steps 5 and 6 of this algorithm.
    
    \item[Step 9.] Same as in Step 8 of Algorithm \ref{procedure1}.
\end{itemize}

\end{algorithm}

\begin{remark}
Notice that, as in Remark \ref{range_qi_times} and Remark \ref{ydiscret2}, we may not need to know $y_{(x_0,\theta_0)}$ in some interval $[a,b)$, but only its value and the values of its derivatives (each component $y_{(x_0,\theta_0),i}$ up to the $\max\{d_i,d_i'+\tilde{q}-1\}-$th derivative) at one time such that Step 5 in Algorithm \ref{procedure3} is satisfied. Moreover, if $0\in\mathcal{S}$ and $y_{(x_0,\theta_0)}$ and its derivatives are known at $t=0$, then we could choose $\tilde{t}=0$ if Step 5 in Algorithm \ref{procedure3} is satisfied and not need to integrate backwards.
\end{remark}

\begin{remark}
    In the algorithms we present here, one can choose a suitable set of time instants $\{t_{j,1},\dots,t_{j,q_j}\}$, $j\in\{1,\dots,p\}$, for Algorithms \ref{procedure1} and \ref{procedure2}, and almost any time $\tilde{t}$ for Algorithm \ref{procedure3} that fulfill the required properties to reconstruct univocally the parameters and initial condition. Then, one can integrate the system up to the current time $t$ to reconstruct the current state. {Reconstructing parameters and initial condition or parameters and current state are equivalent as already mentioned, provided the measurements to be error free.} However, when facing real data corrupted by some noise, the estimation might depend on the choice of the time instants $\{t_{j,\ell}\}_{\ell=1}^{q_j}$, for $j\in\{1,\dots,p\}$, or the time $\tilde{t}$, and is tainted by some error. Then, integrating the system backwards up to time $0$ and forward up to time $t$ might propagate and amplify the estimation error on the initial condition and the current state, respectively. A way to improve this estimation is to use a filter in the spirit of an observer (see e.g. \cite{Kalman}). In this work, as the objective is to analyze the identifiability and observability properties, we do not consider here the robustness issue with respect to corrupted data. This will be the matter of a future work.
\end{remark}

Up to now, we have provided some hypotheses that ensure a system to be observable, identifiable or jointly observable-identifiable, and some constructive algorithms to recover the initial condition and/or the parameter vector. Besides, if we have some analyticity properties, we can use almost any time $t\in\mathcal{I}$ to recover these unknowns. Epidemiological models based on autonomous ODEs are typically analytic systems. In particular, if $f$ is a combination of linear and bilinear terms of the state variables, then $f$ is clearly analytic in all $\mathbb{R}^n$. Due to the Cauchy-Kovalevskaya theorem for ODEs (see \cite{Kepley2021}), each initial condition provides a unique locally analytic solution. In an autonomous system, any point of a solution can be considered as the initial condition of the same solution, and hence this solution is analytic everywhere. Moreover, in the example presented in Section \ref{motivation}, the output is $y=kI$, which, since $I$ is analytic, is also analytic.

\begin{remark}
    We can apply the proposed algorithms to some systems with piece-wise constant parameters. If the parameter vector is $\theta_s$ in an interval $[t_s,t_{s+1})$, $s\in\{0,1,2,\dots\}$, and we know $t_s$, for all $s$, then we can check the different hypotheses exposed in Algorithms \ref{procedure1}, \ref{procedure2} and \ref{procedure3} considering for each $s$ that we know $y_{(x(t_s),\theta_s)}$ in $[t_s,t_{s+1})$, and apply them respectively if the hypotheses are fulfilled.
\end{remark}

Let us now illustrate all the presented theory through the classical SIRS model.

\section{Application to the SIRS model}\label{obsnident}

In this section, we study the observability, identifiability and joined observability-identifiability of the SIRS model \eqref{motiSIRSkI}, previously presented in Section \ref{motivation}, under the observation of an unknown fraction $k\in(0,1]$ of individuals. The associated ODE system is conservative, i.e., $\dot{S}+\dot{I}+\dot{R}=0$; in particular, $S$ is the fraction of susceptible individuals, $I$ is the fraction of infectious individuals and $R$ is the fraction of recovered individuals, and then $S+I+R=1$, for all $t\geq 0$ (see, for example, \cite[Chapter 10.4]{CastilloBrauer}). Then, we consider the following reduced system: \begin{equation}\label{SIRSkI}\left\{\begin{array}{ccl}
    \dot{S}&=&-\beta SI+\mu (1-S-I),\\
    \dot{I}&=&\beta SI-\gamma I, \\
    y&=& kI,
\end{array} \quad \forall\, t\geq 0,\right.\end{equation}for some initial condition $(S(0),I(0))^{\mathrm{T}}=(S_0,I_0)^{\mathrm{T}}\in\Omega=\{(\xi_1,\xi_2)^{\mathrm{T}}\in[0,1]^2 \ : \ \xi_1+\xi_2\leq 1\}$. We assume that the four parameters, $k$, $\beta$, $\gamma$ and $\mu$, and/or the initial condition are unknown.

In this system, $\beta>0$ (days$^{-1}$) is the disease contact rate, $\gamma>0$ (days$^{-1}$) is the transition rate from compartment $I$ to compartment $R$, and $\mu>0$ (days$^{-1}$) is the transition rate from compartment $R$ to compartment $S$. The following results are well-known (\cite[Chapter 10.4]{CastilloBrauer}, \cite{Khalil}):
\begin{itemize}
    \item The system of ODEs of System \eqref{SIRSkI} with initial condition in $\Omega$ has a unique solution.
    \item The set $\Omega$ is positively invariant with respect to the system of ODEs of System \eqref{SIRSkI}.
\item The basic reproduction number is defined as $\mathcal{R}_0=\beta/\gamma$. It is an approximation of the number of cases that one infected person generates on average over the course of its infectious period, in an uninfected population and without special control measures.
\item The (unique) disease-free equilibrium (DFE), which is $P_{\mathrm{f}}=(1,0)^{\mathrm{T}},$ is globally asymptotically stable when $\mathcal{R}_0<1$ and unstable when $\mathcal{R}_0>1$; in particular, in this case, it is a saddle point whose stable manifold is $I\equiv0$.
\item When $\mathcal{R}_0>1$, there exists an endemic equilibrium (EE) in $\Omega$ given by $P_{\mathrm{e}}(\theta)=\left(\dfrac{1}{\mathcal{R}_0},\right.$ $\left.\mu\dfrac{1-1/\mathcal{R}_0}{\gamma+\mu}\right)^{\mathrm{T}},$ which is globally asymptotically stable out of $I\equiv 0$.
\end{itemize}

Having done this quick wrap-up on the main characteristics of the SIRS model, we may start to study if it is observable, identifiable or jointly observable-identifiable. Notice that, in particular, the solutions of this system are analytic and, hence, so it is the output.

\subsection{Identifiability and observability}\label{obsnidentsubsec}

Notice that the solutions and observations of System \eqref{SIRSkI} are analytic in $\mathcal{I}$. We will see that the procedure described in Algorithm \ref{procedure2} is appropriate for this case, and gives more general results than Algorithm \ref{procedure1}. Hence, in the following, we are going to focus on Algorithm \ref{procedure2} considering we know $y_{(x_0,\theta_0)}$ in $[a,b)\subset\mathcal{I}$. In particular, to check the observability, identifiability or joined observability-identifiability of System \eqref{SIRSkI}, we are going to follow Steps 1 and 2 and see if the required conditions can be satisfied. First of all, notice that the number of state variables of  System \eqref{SIRSkI} is $n=2$, the number of unknown parameters is $b=4$ and the number of observations is $m=1$. Let $\xi=(S_0,I_0)^{\mathrm{T}}\in\Omega$, $\theta=(k,\beta,\gamma,\mu)^{\mathrm{T}}\in\Theta=(0,1]\times(0,\infty)^3$, $x=(S,I)^{\mathrm{T}}$. Then, System \eqref{SIRSkI} can be consider as a system of the form given in \eqref{general_sys}, with \begin{equation*}f(x,\theta)=\left(\begin{array}{c}
    -\beta x_1x_2+\mu(1-x_1-x_2)  \\
    \beta x_1x_2-\gamma x_2
\end{array}\right), \quad h(x,\theta)=kx_2.\end{equation*}

In the following, we will find different suitable sets $\Omega_1$ and $\Omega_{2,\theta}$, for any $\theta\in\Theta$, for the above mentioned Steps 1 and 2, respectively. We will do it constructively.

\noindent\textbf{Step 1.}
In this case, since $m=1$, we have to find $d_1\in\mathbb{N}_0$ such that, for any $\theta\in\Theta$, the following function is injective in some suitable set $\Omega_1\subset\Omega$ positively invariant w.r.t. the ODE system given in \eqref{SIRSkI}:
\begin{equation*}\mathcal{L}_{\theta,d_1}: \xi\mapsto \left(y_{(\xi,\theta)}^{(0)}(0), \dots, y_{(\xi,\theta)}^{(d_1)}(0) \right).\end{equation*}

For any $(\xi,\theta)\in\Omega\times\Theta$, letting $x(t;\xi,\theta)=(S(t;\xi,\theta),I(t;\xi,\theta))^{\mathrm{T}}$ be the solution to System \eqref{SIRSkI} with initial condition $\xi$ and parameter vector $\theta$, we have \begin{equation*}y^{(0)}_{(\xi,\theta)}(t) = h_{\theta}(x(t;\xi,\theta)) = kI(t;\xi,\theta),\quad \forall \, t\geq 0.\end{equation*}

Notice that $\mathcal{L}_{\theta,0}(\xi)=y^{(0)}_{(\xi,\theta)}(0)=k\xi_2$ is injective in a set $\Omega_1\subset\Omega$ if, and only if, there exist $\Xi\subset[0,1]$ and some function $G:\Xi\rightarrow [0,1]$ such that $\Omega_1=\{(G(\xi_2),\xi_2)\in\Omega \ : \ \xi_2\in\Xi\}$. A set $\Omega_1$ that has this form is positively invariant with respect to the system of ODEs of System \eqref{SIRSkI} if, and only if, for any initial condition $\xi\in\Omega_1$, $S(t;\xi,\theta)=G(I(t;\xi,\theta))$, for all $t\geq 0$. Two examples of such orbits are $(S,I)\equiv P_{\mathrm{f}}$ and, whenever $\mathcal{R}_0>1$, $(S,I)\equiv P_{\mathrm{e}}(\theta)$.

These type of sets are very restrictive and not very useful in applications. Therefore, let us consider $d_1=1$ instead of $d_1=0$. We will show that this is a more convenient choice since the positively invariant set w.r.t. the ODE system given in \eqref{SIRSkI} that we will derive from it is less restrictive.

Differentiating $y_{(\xi,\theta)}^{(0)}(t)$ one time leads to the following equation:
\begin{equation*}\begin{array}{ll}
y^{(1)}_{(\xi,\theta)}(t)&=\mathrm{D}h_{\theta}(x(t;\xi,\theta))f_{\theta}(x(t;\xi,\theta))=(0,k)f_{\theta}(x(t;\xi,\theta))\\ &=k(\beta S(t;\xi,\theta)-\gamma)I(t;\xi,\theta)(t)=(\beta S(t;\xi,\theta)-\gamma)y^{(0)}_{(\xi,\theta)}(t),\quad \forall \, t\geq 0.
\end{array}\end{equation*}

Let us consider $\Omega_1=\{(\xi_1,\xi_2)^{\mathrm{T}}\in\Omega \ : \ \xi_2\neq 0\}=\{(\xi_1,\xi_2)^{\mathrm{T}}\in[0,1)\times(0,1] \ : \ \xi_1+\xi_2\leq 1\}$.

\begin{lemma}\label{lemmaLtheta1}For any $\theta=(k,\beta,\gamma,\mu)^{\mathrm{T}}\in\Theta$, the following map is injective: \begin{equation}\label{Ltheta1}\mathcal{L}_{\theta,1}: 
    \Omega_1  \longrightarrow  \mathbb{R}^2 :
      \xi  \longmapsto  \left(y^{(0)}_{(\xi,\theta)}(0),y^{(1)}_{(\xi,\theta)}(0)\right).
\end{equation}
\end{lemma}
\begin{proof}
    The proof is straightforward, since, given $\theta=(k,\beta,\gamma,\mu)\in\Theta$, $\xi\in\Omega_1$, \begin{equation*}y^{(0)}_{(\xi,\theta)}(0)=k\xi_2,\quad y^{(1)}_{(\xi,\theta)}(0)=k(\beta\xi_1-\gamma)\xi_2.\end{equation*}Therefore, it is easy to see that \eqref{Ltheta1} is injective.
\end{proof}
Notice that, if we change $\Omega_1$ by $\Omega$ in Lemma \ref{lemmaLtheta1}, then the mapping $\mathcal{L}_{\theta,1}$ is not injective. 

To continue with Step 1 of Algorithm \ref{procedure2}, we point out that $\Omega_1$ is positively invariant with respect to the system of ODEs of System \eqref{SIRSkI}. Notice that $\Omega_1$ is the same set as $\Omega$ except for the manifold $\{(\xi_1,0):\xi_1\in[0,1]\}$, which is also positively invariant w.r.t. the ODE system given in \eqref{SIRSkI} (recall that this is the stable manifold associated to $P_{\mathrm{f}}$ when $\mathcal{R}_0>1$). Then, due to the uniqueness of solutions, $\Omega_1$ is also positively invariant w.r.t. the ODE system given in \eqref{SIRSkI}. Hence, System \eqref{SIRSkI} is observable on $\Omega_1$ in any semi-open interval $[a,b)\subset\mathcal{I}$ with parameters in $\Theta$. If we know $\theta_0\in\Theta$ and some observations $y_{(x_0,\theta_0)}(t)$, for all $t\in [a,b)$, then we can determine $x_0$ performing now Steps 3, 7 and 8 in Algorithm \ref{procedure2}.

\noindent\textbf{Step 2.} We now need to find a suitable sets $\Omega_{2,\theta}\subset\Omega$, for every $\theta\in\Theta$, positively invariant w.r.t. the ODE system given in \eqref{SIRSkI}, and maps $g:\mathcal{D}\rightarrow\mathbb{R}^{q+p}$ and $r:\Theta\rightarrow\mathbb{R}^p$, for some suitable $\mathcal{D}\subset\mathbb{R}^{d_1'+1}$, $d_1'\in\mathbb{N}\cup\{0\}$, $q,p\in\mathbb{N}$, such that (C1), (C2) and (C3) of Theorem \ref{generalization_ident} are satisfied for some connected $\mathcal{S}\subset\mathcal{I}$ such that $[a,b)\subset\mathcal{S}$ and $g_{j,l}\left(y_{(\xi,\theta)}^{(0)}(t),\dots,y_{(\xi,\theta)}^{(d_1')}(t)\right)$ are analytic w.r.t. $t\in\mathcal{S}$, for any $j\in\{1,\dots,p\}$, $l\in\{1,\dots,q_j\}$, and {$(\xi,\theta)\in\Gamma_{2,\Theta}$}.

In the following, for the sake of cleanliness, we make a slight abuse of notation and avoid the specification of $\xi$, $\theta$ and $t$ when no confusion is possible; besides, we denote $y^{(0)}$ and $y^{(1)}$ as $y$ and $\dot{y}$, respectively.

Let us start by checking (C1) of Theorem \ref{generalization_ident}. Since \begin{equation}\label{ydoty}y=kI,\quad \dot{y}=(\beta S-\gamma)y,\quad \forall \ t\geq 0,\end{equation}neither $d_1'=0$ nor $d_1'=1$ are suitable. In particular, we cannot obtain a function $r$ injective in $\Theta$ suitable for \eqref{linear_eq} since we do not have any information on $\mu$ in the expressions of $y^{(0)}$ and $y^{(1)}$ given in \eqref{ydoty}.

Then, if we continue differentiating $\dot{y}$, denoting $y^{(2)}$ as $\ddot{y}$, we obtain: \begin{equation}\label{ddoty}\ddot{y}=(\beta S-\gamma)\dot{y}+\beta\dot{S}y=(\beta S-\gamma)y+\beta (-\beta SI+\mu(1-S-I))y, \quad \forall \ t\geq 0.\end{equation}Notice that, if $y\not\equiv 0$, which holds when we consider initial conditions in $\Omega_1$, we can define $(S,I)$ in terms of $y$, $\dot{y}$ and $\theta$ from \eqref{ydoty} as \begin{equation}\label{exprSI}S=\dfrac{1}{\beta}\left(\dfrac{\dot{y}}{y}+\gamma\right),\quad I=\dfrac{y}{k},\quad \forall \ t\geq 0.\end{equation}Hence, substituting \eqref{exprSI} in \eqref{ddoty} and performing some computations, we obtain an equation in the form of \eqref{linear_eq} in (C1) of Theorem \ref{generalization_ident}: \begin{equation}\label{linear_eq_SIRSkI} \ddot{y}-\dfrac{\dot{y}^2}{y}=-y\mu(\gamma-\beta)-y^2\dfrac{\beta}{k}(\gamma+\mu)-\dot{y}\mu-y\dot{y}\dfrac{\beta}{k},\quad \forall \ t\geq 0.\end{equation}
Then, we can see that (C1) of Theorem \ref{generalization_ident} is satisfied with $d_1'=2$, $p=1$, $q_1=q=4$, \begin{equation*}g_0(y,\dot{y},\ddot{y})=\ddot{y}-\dfrac{\dot{y}^2}{y}, \ g_1(y,\dot{y},\ddot{y})=-y, \ g_2(y,\dot{y},\ddot{y})=-y^2, \ g_3(y,\dot{y},\ddot{y})=-\dot{y}, \ g_4(y,\dot{y},\ddot{y})=-y\dot{y},\end{equation*} and \begin{equation*}r(\theta)=\left(\mu(\gamma-\beta),\dfrac{\beta}{k}(\gamma+\mu),\mu,\dfrac{\beta}{k}\right),\quad \forall\,\theta\in\Theta.\end{equation*}
Since $p=1$, we have denoted $g_{1,l}=g_l$, $l\in\{0,\dots,4\}$, for the sake of simplicity.

Let us now check (C2) of Theorem \ref{generalization_ident}, i.e., if the function $r:\Theta\rightarrow r(\Theta)\subset\mathbb{R}^4$ is injective. Indeed, given some parameter vector $\theta=(k,\beta,\gamma,\mu)^{\mathrm{T}}\in\Theta$, it is easy to see that we can invert the equation $r(\theta)=\sigma$ and obtain univocally \begin{equation*}\left(k,\beta,\gamma,\mu\right)=\left(\dfrac{1}{\sigma_4}\left(\dfrac{\sigma_2}{\sigma_4}-\sigma_3-\dfrac{\sigma_1}{\sigma_3}\right),\dfrac{\sigma_2}{\sigma_4}-\sigma_3-\dfrac{\sigma_1}{\sigma_3},\dfrac{\sigma_2}{\sigma_4}-\sigma_3,\sigma_3\right).\end{equation*}

Finally, we need to check (C3) of Theorem \ref{generalization_ident}, i.e., if, for any {$(\xi,\theta)\in\Gamma_{2,\Theta}$}, the functions $\{g_i(y,\dot{y},\ddot{y})\}_{i=1}^4$ are linearly independent with respect to time in some suitable $\mathcal{S}\subset\mathcal{I}$, for initial conditions in some suitable sets $\Omega_{2,\theta}\subset\Omega$, for $\theta\in\Theta$, positively invariant w.r.t. the ODE system given in \eqref{SIRSkI}. Since these functions are defined in all $\mathcal{I}$, we will consider $\mathcal{S}=\mathcal{I}$. Then, if they are linearly independent in $\mathcal{I}$, given that they are clearly also analytic, we will have linear independence in any $[a,b)\subset\mathcal{I}$. In the following Lemma \ref{li_SIRSkI}, we will directly prove linear independence for any $[a,b)\subset\mathcal{I}$.
\begin{lemma}\label{li_SIRSkI}
    For any $[a,b)\subset\mathcal{I}$, the functions $\{g_i(y,\dot{y},\ddot{y})\}_{i=1}^4$ are linearly independent in $[a,b)$ if {$\Gamma_{2,\Theta}=\bigcup_{\theta\in\Theta} (\Omega_1\setminus\{P_{\mathrm{e}}(\theta)\})\times\{\theta\}$ and $(\xi,\theta)\in \Gamma_{2,\Theta}$}, i.e., $\Omega_{2,\theta}=\Omega_1\setminus\{P_{\mathrm{e}}(\theta)\}$, where $P_{\mathrm{e}}(\theta)$ denotes the endemic equilibrium point associated with $\theta$.
\end{lemma}
\begin{proof}
Consider some $[a,b)\subset\mathcal{I}$. For some $\theta\in\Theta$, we consider $\xi\in\Omega$, and assume there exist $a_1,a_2,a_3,a_4\in\mathbb{R}$ such that \begin{equation}\label{a1234y}a_1y(t)+a_2y(t)^2+a_3\dot{y}(t)+a_4y(t)\dot{y}(t)=0, \quad \forall \ t\in[a,b).\end{equation}We will prove that $a_1=a_2=a_3=a_4=0$ if $\xi\in\Omega_{2,\theta}$, for some $\Omega_{2,\theta}\subset\Omega$. As mentioned in Step 1, if $\xi_2=0$, then $y\equiv 0$ in $\mathcal{I}$, and hence \eqref{a1234y} is true for values $a_1,a_2,a_3,a_4$ not necessarily satisfying $a_1=a_2=a_3=a_4=0$, which implies that the corresponding functions $\{g_i(y,\dot{y},\ddot{y})\}_{i=1}^4$ are not linearly independent in $[a,b)$.

Therefore, let us consider $\xi_2\neq 0$, i.e., $\xi\in\Omega_1$. In this case, if $(S(t),I(t))^{\mathrm{T}}$ is the solution of the ODE system of \eqref{SIRSkI} with initial condition $\xi$ and parameter vector $\theta$, recall that $\Omega_1$ is positively invariant w.r.t. the ODE system given in \eqref{SIRSkI}, i.e., $(S(t),I(t))^{\mathrm{T}}\in\Omega_1$, for all $t\geq 0$. Then, we can rewrite \eqref{a1234y} as \begin{eqnarray*}a_1kI(t)+a_2k^2I^2(t)+a_3kI(t)(\beta S(t)-\gamma)+a_4k^2I^2(t)(\beta S(t)-\gamma)&=&0 \implies \\ \overset{(S,I)^{\mathrm{T}}\in\Omega_1}{\implies} a_1+a_2kI(t)+a_3(\beta S(t)-\gamma)+a_4kI(t)(\beta S(t)-\gamma)&=&0, \quad \forall\, t\in[a,b),\end{eqnarray*}which leads to \begin{equation*}(a_1-a_3\gamma)+a_3\beta S(t)+(a_2k-a_4k\gamma)I(t)+a_4k\beta S(t)I(t)=0, \quad \forall\, t\in [a,b).\end{equation*}Determining $\{a_i\}_{i=1}^4$ not all of them null such that this is fulfilled is equivalent to determining $A,B,C,D$ not all of them null such that \begin{equation}\label{ldepSI}A+BS(t)+CI(t)+DS(t)I(t)=0, \quad t\in[a,b),\end{equation}since $A=B=C=D=0$ if, and only if, $a_1=a_2=a_3=a_4=0$. That is, we will check whether $1,S,I,SI$ are linearly independent functions in $[a,b)$ considering initial conditions in $\Omega_1$ or not.

If we consider that $S$ or $I$ are constant, \eqref{ldepSI} is true for values of $A,B,C,D$ not necessarily satisfying $A=B=C=D=0$, which implies that the corresponding functions $\{g_i(y,\dot{y},\ddot{y})\}_{i=1}^4$ are not linearly independent in $[a,b)$. In $\Omega_1$, due to the analyticity of the system, this is only possible if $(S,I)^{\mathrm{T}}\equiv P_{\mathrm{e}}(\theta)$. Indeed, $\dot{I}\equiv0$ implies $I\equiv0$, which is excluded from $\Omega_1$, or $S\equiv \beta/\gamma$ and hence $(S,I)^{\mathrm{T}}\equiv P_{\mathrm{e}}(\theta)$; on the other hand, if $\dot{S}\equiv 0$, then, $S\equiv \xi_1$ and \begin{equation*}I\equiv\dfrac{\mu(1-\xi_1)}{\mu+\beta \xi_1},\quad \forall\, t\in[a,b),\end{equation*}i.e., we are at an equilibrium point. 

Hence, let us take $\Omega_{2,\theta}=\Omega_1\setminus\{P_{\mathrm{e}}(\theta)\}$ and check if, when $\xi\in\Omega_{2,\theta}$, the corresponding functions $1,S,I,SI$ are linearly independent in $\Omega_{2,\theta}$ for times in $[a,b)$. Assume that \eqref{ldepSI} holds and at least one between $C$ and $D$ is non-null. If we consider initial conditions in $\Omega_{2,\theta}$, we obtain the following expression for $I$: \begin{equation*}I(t)=-\dfrac{A+BS(t)}{C+DS(t)},\end{equation*}where $C+DS(t)$ is non-null for almost every $t\in [a,b)$, since $S$ is analytic and non-constant in $\Omega_{2,\theta}$. Without loss of generality, assume $C+DS(t)$ is non-null for all $t\in[a,b)$. We can hence differentiate this expression for $I(t)$ and obtain \begin{equation*}\dot{I}(t)=-\dot{S}(t)\dfrac{BC-AD}{(C+DS(t))^2} =\left((\beta S(t)+\mu)\dfrac{A+BS(t)}{C+DS(t)}-\mu(1-S(t))\right)\dfrac{BC-AD}{(C+DS(t))^2}, \quad \forall\, t\in [a,b).\end{equation*}On the other hand, \begin{equation*}\dot{I}(t)=\beta S(t)I(t)-\gamma I(t)=(\gamma-\beta S(t))\dfrac{A+BS(t)}{C+DS(t)}, \quad \forall\, t\in[a,b).\end{equation*}If we equal both expressions for $\dot{I}(t)$, we get \begin{equation*}\left((\beta S(t)+\mu)\dfrac{A+BS(t)}{C+DS(t)}-\mu(1-S(t))\right)\dfrac{BC-AD}{(C+DS(t))^2}=(\gamma-\beta S(t))\dfrac{A+BS(t)}{C+DS(t)},\end{equation*}for $t\in[a,b)$. After some computations in order to get rid of the denominators, we reach the following polynomial on $S(t)$: \begin{equation*}c_4S(t)^4+c_3S(t)^3+c_2S(t)^2+c_1S(t)+c_0=0,\quad \forall\, t\in[a,b),\end{equation*}where \begin{eqnarray*}c_0&=&\mu ABC-\mu A^2D+\mu ACD-\gamma AC^2-\mu BC^2,\\ c_1&=&\mu AD^2-\beta A^2D-(2\gamma+\mu)ACD-\mu ABD-\mu BCD+\beta AC^2-(\gamma-\mu)BC^2+\mu B^2C+\beta ABC, \\
c_2&=&\beta BC^2-(2\gamma-\mu)BCD-(\gamma+\mu)AD^2-\beta ABD+2\beta ACD+\beta B^2C, \\ c_3&=&\beta AD^2-\gamma BD^2+2\beta BCD, \\ c_4&=&\beta BD^2.\end{eqnarray*} Since $S$ is analytic non-constant in $[a,b)$, for this polynomial to be 0 in $[a,b)$, we need that $c_0=\dots=c_4=0$. We have assumed that at least one between $C$ and $D$ is non-null. We have two subcases:
\begin{itemize}
    \item Assume $D\neq 0$. Then, it is straightforward observing that $c_3=c_4=0$ implies $A=B=0$. But then, this implies that $I\equiv 0$ in $[a,b)$, which cannot happen for solutions in $\Omega_{2,\theta}$. Hence, $D=0$.
    \item Assume $C\neq 0$, and $D=0$ due to the previous argument. Then, $c_2=0$ implies $B=0$ or $B=-C$. If $B=0$, from $c_1=0$, we obtain $A=0$, which again cannot happen for solutions in $\Omega_{2,\theta}$. If we consider $B=-C$, $c_1=0$ implies $\gamma C^3=0$, which is not true since $\gamma>0$ and we have assumed $C\neq 0$.
\end{itemize}
Hence, we must have $C=D=0$. Then, equation \eqref{ldepSI} is \begin{equation*}A+BS(t)=0,\quad \forall\, t\in[a,b).\end{equation*}But $S$ is analytic non-constant for solutions in $\Omega_{2,\theta}$, and hence $A=B=0$.

Therefore, for any $\theta\in\Theta$, $1,S,I,SI$ or, equivalently, $\{g_i(y,\dot{y},\ddot{y})\}_{i=1}^4$ are linearly independent in any semi-open subinterval $[a,b)\subset\mathcal{I}$ if we consider $\xi$ in the following set: \begin{equation*}\Omega_{2,\theta}=\left\{(\xi_1,\xi_2)^{\mathrm{T}}\in[0,1)\times (0,1] \ : \ \xi_1+\xi_2\leq 1\right\}\setminus \{P_{\mathrm{e}}(\theta)\}\subset\Omega_1\subset\Omega.\end{equation*}
\end{proof}

To continue with Step 2, we need to check if the set $\Omega_{2,\theta}$ is positively invariant with respect to the system of ODEs of System \eqref{SIRSkI}.

\begin{lemma}\label{Omega_posinv}
    Let $\theta\in\Theta$. The set $\Omega_{2,\theta}=\{(\xi_1,\xi_2)^{\mathrm{T}}\in[0,1)\times (0,1] \ : \ \xi_1+\xi_2\leq 1\}\setminus\{P_{\mathrm{e}}(\theta)\}$ is positively invariant with respect to the system 
\begin{equation}\label{SIRS}\left\{\begin{array}{ccl}
    \dot{S}&=&-\beta SI+\mu(1-S-I),\\
    \dot{I}&=&\beta SI-\gamma I.
\end{array}\right.\end{equation}
\end{lemma}
\begin{proof}
   As said at the beginning of Section \ref{obsnident} the set \begin{equation*}\Omega=\{(\xi_1,\xi_2)^{\mathrm{T}}\in[0,1]^2\, : \, \xi_1+\xi_2\leq1\}\end{equation*} is positively invariant with respect to System \eqref{SIRS}.
    
    Let $P_{\mathrm{f}}=(1,0)^{\mathrm{T}}$ be the disease-free equilibrium of System \eqref{SIRS} and, given $\theta\in\Theta$, $P_{\mathrm{e}}(\theta)=(1/\mathcal{R}_0, \mu(1-1/\mathcal{R}_0)/(\mu+\gamma))^{\mathrm{T}}$ the endemic equilibrium (which is inside $\Omega$ if, and only if,  $\mathcal{R}_0>1$). Since the solution of this system given an initial condition is unique in $t\geq 0$, then the set $\Omega\setminus\{P_{\mathrm{f}},P_{\mathrm{e}}(\theta)\}$ is still positively invariant with respect to System \eqref{SIRS}.

    On the other hand, if an initial condition $\xi$ is such that $\xi_2\neq 0$, then $I(t)>0$, for all $t\geq 0$, due to the uniqueness of solutions. Indeed, since $S(t)\geq 0$, $\forall\, t\geq 0$, \begin{equation*}\dot{I}(t)=(\beta S(t)-\gamma)I(t)\geq -\gamma I(t),\quad \forall\, t\geq 0 \quad \implies \quad I(t)\geq \xi_2\mathrm{e}^{-\gamma t}>0,\quad \forall\, t\geq0.\end{equation*}
    Thus, $\{(\xi_1,\xi_2)^{\mathrm{T}} \in[0,1)\times (0,1] \ : \ \xi_1+\xi_2\leq 1\}\setminus\{P_{\mathrm{e}}(\theta)\}$ is also a positively invariant set with respect to System \eqref{SIRS}.
\end{proof}

Hence, Step 2 is finished, which allows us to conclude that System \eqref{SIRSkI} is identifiable on $\Theta$ in any semi-open interval $[a,b)\subset\mathcal{I}$ with initial conditions {consistent with the family $\{\Omega_{2,\theta}\}_{\theta\in\Theta}$}. Actually, if we know $x_0\in\Omega_{2,\theta_0}$  and some observations $y_{(x_0,\theta_0)}(t)$, for $t\in [a,b)\subset\mathcal{I}$, then we can determine $\theta_0$ performing Steps 3--6 in Algorithm \ref{procedure2}.

\textbf{Step 3.} Set $\Omega_{\theta}=\Omega_1\cap\Omega_{2,\theta}=\Omega_{2,\theta}$, $\theta\in\Theta$. Then, this is a positively invariant set, and we assume $x_0\in\Omega_{\theta_0}$.

Notice that, for any $\theta\in\Theta$, not considering initial points in $\{(\xi_1,0):\xi_1\in[0,1]\}\cup\{P_{\mathrm{e}}(\theta)\}$ is not restrictive with respect to the observations we can work with, since $y\equiv 0$ would mean there is no disease and therefore no point of study, and $y\equiv$constant$\neq0$ means the disease is already endemic and it will not vary unless we perturb the system (e.g., with migration or vaccination).

Hence, System \eqref{SIRSkI} is jointly observable-identifiable on {$\Gamma_{\Theta}$} in any $[a,b)\subset\mathcal{I}$. If we had some particular observations $y_{(x_0,\theta_0)}$ in $[a,b)\subset\mathcal{I}$, we could continue with Steps 4--8 in Algorithm \ref{procedure2} or Steps 4--9 in Algorithm \ref{procedure3}.

\subsection{A limiting case: The SIR model}\label{particularSIR}

Another basic compartmental model is the SIR model, where the considered populations are the same as for System \eqref{SIRSkI}, but there is no loss of immunity, i.e., $\mu=0$. Again, we can consider the observation of an unknown portion of infected individuals and hence obtain the following system:
\begin{equation}\label{SIRkI}
    \left\{\begin{array}{lcl}
        \dot{S} & = & -\beta SI, \\
        \dot{I} & = & \beta SI-\gamma I, \\
        y & = & kI,
    \end{array}\right.
\end{equation}where we have already simplified the dimension taking into account that the quantity $S+I+R=1$ is preserved.

If one thinks of whether System \eqref{SIRkI} is observable, identifiable or jointly observable-identifiable, it seems it will satisfy the same properties as the previous System \eqref{SIRSkI}, and it will be easier to prove them, since we got rid of parameter $\mu$. However, this is far from being true.

If we perform the same computations as for the SIRS model, we obtain the same result for Step 1, which implies that System \eqref{SIRkI} is observable on $\Omega_1=\{(\xi_1,\xi_2)^{\mathrm{T}}\in[0,1)\times(0,1]\ : \ \xi_1+\xi_2\leq 1\}$ in any $[a,b)\subset\mathcal{I}$ with parameters in $\Theta=(0,1]\times(0,\infty)^2$. Regarding Step 2, as in \eqref{exprSI} and \eqref{linear_eq_SIRSkI} with $\mu=0$, we obtain that \begin{equation}\label{SI_SIRkI}I=\dfrac{y}{k},\quad S=\dfrac{1}{\beta}\left(\dfrac{\dot{y}}{y}+\gamma\right)\end{equation}and \begin{equation}\label{ddotSIR}\ddot{y}-\dfrac{\dot{y}}{y}=-\dfrac{\beta\gamma}{k}y^2-\dfrac{\beta}{k}y\dot{y}.\end{equation} Then, \begin{equation*}r(k,\beta,\gamma)=\left(\dfrac{\beta\gamma}{k},\dfrac{\beta}{k}\right),\end{equation*}which is clearly not injective in $\Theta$, and hence we cannot complete Step 2 using \eqref{ddotSIR}. However, recalling Remark \ref{sufficient_conditions}, System \eqref{SIRkI} is indeed identifiable considering a suitable positively invariant subset of $\Omega_1$, as we can see in the following Lemma \eqref{SIRkIident}, which does not need the injectivity of $r$.
\begin{lemma}\label{SIRkIident}
    System \eqref{SIRkI} is identifiable on $\Theta=(0,1]\times(0,\infty)^2$ in any $[a,b)\subset\mathcal{I}$, with initial conditions in $\Tilde{\Omega}=\{(\xi_1,\xi_2)^{\mathrm{T}}\in(0,1)^2\ : \ \xi_1+\xi_2\leq 1\}$\footnote{Note that $\Tilde{\Omega}$ is independent of the choice of $\theta$.}.
\end{lemma}
\begin{proof}
    System \eqref{SIRkI} will be identifiable on $\Theta$ with initial conditions in $\Tilde{\Omega}$ if, given any initial condition $\xi\in\Tilde{\Omega}$ and some observations $y$ in $[a,b)$, we can determine $\theta\in\Theta$ univocally.

    First of all, one can prove, similarly to Lemma \ref{Omega_posinv}, that $\Tilde{\Omega}$ is positively invariant with respect to the ODE system in System \eqref{SIRkI}. Equations \eqref{SI_SIRkI} and \eqref{ddotSIR} are still valid. Therefore, if we prove that $y^2$ and $y\dot{y}$ are linearly independent w.r.t. $t\in[a,b)$, then there is a unique solution $\sigma=(\sigma_1,\sigma_2)$ to \begin{equation}\label{ddotsigma}\ddot{y}-\dfrac{\dot{y}}{y}=-\sigma_1y^2-\sigma_2y\dot{y},\end{equation}where $\sigma_1=\beta\gamma/k$ and $\sigma_2=\beta/k$. To see that the linear independence holds, let us consider $a_1,a_2\in\mathbb{R}$ such that \begin{equation*}a_1y^2+a_2y\dot{y}=0,\quad \forall\ t\in [a,b).\end{equation*}Since $\tilde{\Omega}$ is positively invariant, then $I\not \equiv 0$ in $[a,b)$, and hence $y\not\equiv 0$, which implies \begin{equation*}a_1y+a_2\dot{y}=0.\end{equation*}This is equivalent to \begin{equation*}a_1kI+a_2kI(\beta S-\gamma)=0,\ \forall\ t\in [a,b), \quad \text{which implies}\quad a_1+a_2(\beta S-\gamma)=0, \ \forall\ t\in[a,b).\end{equation*}Since $S$ is analytic, $a_1$ or $a_2$ are non-null if, and only if, $S$ is constant. Given that $\dot{S}=-\beta SI$, this can only happen if $S(t)=0$ or $I(t)=0$, for all $t\geq 0$, which cannot occur in $\Tilde{\Omega}$. Hence, $y^2$ and $y\dot{y}$ are linearly independent when we consider $\xi\in\Tilde{\Omega}$ and $\theta\in\Theta$, and therefore there is a unique solution $\sigma$ to \eqref{ddotsigma}. Then, we can determine univocally $\beta/k=\sigma_2$ and $\gamma=\sigma_1/\sigma_2$.

    If $a=0$, taking into account that $\xi_1,\xi_2\neq 0$ in $\tilde{\Omega}$, we can conclude directly determining \begin{equation*}k=\dfrac{y(0)}{\xi_2},\quad \gamma=\dfrac{\sigma_1}{\sigma_2},\quad \beta=\dfrac{1}{\xi_1}\left(\dfrac{\dot{y}(0)}{y(0)}+\dfrac{\sigma_1}{\sigma_2}\right).\end{equation*}

    If $a>0$, notice now that we can make the following change of variables in the ODE system in System \eqref{SIRkI}: $X=\beta S$ and $Y=k I$. Then, \begin{equation*}\left\{\begin{array}{lcl}\dot{X} & = & -\sigma_2 XY,\\
        \dot{Y} & = & Y(X-\gamma),\end{array}\right.\end{equation*}and we can recover $\beta \xi_1$ and $k\xi_2$ integrating backwards considering as initial condition \begin{equation*}X(a)=\dfrac{\dot{y}(a)}{y(a)}+\dfrac{\sigma_1}{\sigma_2},\quad Y(a)=y(a).\end{equation*}Once we know $X(0)=\beta\xi_1$ and $Y(0)=k\xi_2$, taking into account that $\xi_1,\xi_2\neq 0$ in $\Tilde{\Omega}$, we can determine univocally $k$, $\beta$ and $\gamma$ as \begin{equation*}k=\dfrac{Y(0)}{\xi_2},\quad \beta = \dfrac{X(0)}{\xi_1},\quad \gamma=\dfrac{\sigma_1}{\sigma_2}.\end{equation*}
\end{proof}

Therefore, System \eqref{SIRkI} is observable on $\Omega_1$ with parameters in $\Theta$ and identifiable on $\Theta$ with initial conditions in $\Tilde{\Omega}\subset\Omega_1$, in any $[a,b)\subset\mathcal{I}$. Nevertheless, it is not jointly observable-identifiable on $\Tilde{\Omega}\times\Theta$. Actually, in \cite{Cunniffeetal}, the authors treat this case and prove that, assuming both $x_0=(S_0,I_0)\in\Tilde{\Omega}$ and $\theta_0=(k_0,\beta_0,\gamma_0)\in\Theta$ unknown, we can only determine $\gamma_0$, $\beta_0/k_0$, $\beta_0 S_0$ and $k_0 I_0$, i.e., it is \textit{partially} jointly observable-identifiable. In fact, considering \eqref{SI_SIRkI} and \eqref{ddotSIR}, if $(\sigma_1,\sigma_2)$ is the solution to \begin{equation*}\ddot{y}-\dfrac{\dot{y}}{y}=-\sigma_1y^2-\sigma_2y\dot{y}\end{equation*}(which we know is unique due to the proof of Lemma \ref{SIRkIident}) and $(k_0,\beta_0,\gamma_0)$ is any solution to $r(k_0,\beta_0,\gamma_0)=(\sigma_1,\sigma_2)$, notice that our method matches this result, since we have \begin{equation*}I=\dfrac{y}{k_0},\quad S=\dfrac{1}{\beta_0}\left(\dfrac{\dot{y}}{y}+\gamma_0\right),\quad \dfrac{\beta_0\gamma_0}{k_0}=\sigma_1,\quad \dfrac{\beta_0}{k_0}=\sigma_2,\end{equation*}and hence, if we know $y(0)$ and $\dot{y}(0)$, it is straightforward checking that we can recover $\gamma_0$, $\beta_0/k_0$, $\beta_0S_0$, $k_0I_0$ univocally as follows: \begin{equation*}\gamma_0=\dfrac{\sigma_1}{\sigma_2},\quad \dfrac{\beta_0}{k_0}=\sigma_2,\quad \beta_0S_0=\dfrac{\dot{y}(0)}{y(0)}+\dfrac{\sigma_1}{\sigma_2},\quad k_0I_0=y(0),\end{equation*}and we cannot obtain more information. If $a>0$, then we can proceed as in the proof of Lemma \ref{SIRkIident} performing the change of variables $X=\beta_0S$ and $Y=k_0I$ and integrating backwards from $t_0=a$ to $0$.

One could also not know at first if some given observations $y_{(x_0,\theta_0)}$ correspond to an SIR or an SIRS model when both $x_0$ and $\theta_0$ are unknown, and wonder if the same data might be reproduced with two different sets of parameters $(k_1,\beta_1,\gamma_1,0)^{\mathrm{T}}$ and $(k_2,\beta_2,\gamma_2,\mu_2)^{\mathrm{T}}$, $\mu_2\neq 0$. To tackle this question, we consider an \textit{extended SIRS model} with parameters \begin{equation*}\theta:=(k,\beta,\gamma,\mu)^{\mathrm{T}}\in\Theta^*=(0,1]\times(0,\infty)^2\times [0,\infty).\end{equation*}Notice that the SIR model is a particular case of this extension of the SIRS model, and can be regarded as a limiting case of the SIRS model when $\mu\rightarrow 0$.

Then again, in Step 2 we would obtain \begin{equation*}r(k,\beta,\gamma,\mu)=\left(\mu(\gamma-\beta),\dfrac{\beta}{k}(\gamma+\mu),\mu,\dfrac{\beta}{k}\right),\end{equation*}which is not injective in $\Theta^*$ and hence we cannot complete this step with this function $r$. Nevertheless, this extended model can help to distinguish whether the observations come from an SIR model or an SIRS model, as illustrated in Section \ref{distinguish}. Let us first do a quick comparison between both models in Section \ref{compSIRS}.

\subsubsection{Comparison between SIR and SIRS models} \label{compSIRS}

Although passing from $\mu= 0$ to $\mu> 0$ may change substantially the behavior of the solutions, since the SIR model does not admit an endemic state, whereas the SIRS model does, they can be hardly distinguishable at early stages if $\mu$ is very small. We let $x_{\mathrm{SIR}}(t)$ and $x_{\mathrm{SIRS}}(t)$ be the solutions to the SIR given by the ODE system of \eqref{SIRkI} and the SIRS model \eqref{SIRS}, respectively, with the same initial condition, and we study the dependence of $||x_{\mathrm{SIR}}(t)-x_{\mathrm{SIRS}}(t)||_2$ on $\mu$. To do this, we are going to consider the same parameters $\beta$ and $\gamma$ for both models, and we will base ourselves on Theorem 3.4., Chapter 3 of \cite{Khalil}, particularized to our autonomous context:

\begin{theorem}{\cite[Chapter 3]{Khalil}}\label{KhalilThm}
    Let $f$ be a Lipschitz map on $W$ with a global Lipschitz constant $L$, where $W\subset \mathbb{R}^n$ is an open connected set. Let $y(t)$ and $z(t)$ be solutions of \begin{equation*}\dot{y}=f(y),\ y(t_0)=y_0,\quad \text{and}\quad \dot{z}=f(z)+g(z),\ z(t_0)=z_0,\end{equation*}such that $y(t)$, $z(t)\in W$, for all $t\in[t_0,t_1]$. Suppose that \begin{equation*}||g(x)||\leq \mu,\quad \forall\, x\in W,\end{equation*}for some $\mu>0$. Then, \begin{equation*}||y(t)-z(t)||\leq ||y_0-z_0||\exp(L(t-t_0))+\dfrac{\mu}{L}[\exp(L(t-t_0))-1],\quad \forall\, t\in[t_0,t_1].\end{equation*}
\end{theorem}

Then, we are going to check the different conditions required in Theorem \ref{KhalilThm} in order to obtain an estimation of $||x_{\mathrm{SIR}}(t)-x_{\mathrm{SIRS}}(t)||_2$.

Notice that both models share the same positively invariant set \begin{equation*}\Omega=\{(\xi_1,\xi_2)^{\mathrm{T}}\in[0,1]^2 \ : \ \xi_1+\xi_2\leq 1\},\end{equation*}which is compact, but the derivatives are well defined. Let $x_{\mathrm{SIR}}=(S_{\mathrm{SIR}},I_{\mathrm{SIR}})^{\mathrm{T}}$ and $x_{\mathrm{SIRS}}=(S_{\mathrm{SIRS}},I_{\mathrm{SIRS}})^{\mathrm{T}}$ be the state variables associated to the ODE system of \eqref{SIRkI} and \eqref{SIRS}, respectively. We are going to rewrite these systems in the form presented in Theorem \ref{KhalilThm}. Let $x_0$ be the same initial condition at $t_0$ for both models, $\eta=(\beta,\gamma,\mu)^{\mathrm{T}}$ and $\nu=(\beta,\gamma)^{\mathrm{T}}$ two parameter vectors, and the following function: \begin{equation*}f_{\mathrm{SIR}}(x,\nu)=\left(\begin{array}{c}
    -\beta x_1x_2  \\
     \beta x_1x_2-\gamma x_2
\end{array}\right), \quad x\in\Omega.\end{equation*}Notice that $f_{\mathrm{SIR}}(x,\nu)$ is Lipschitz in $x$ on the compact set $\Omega$, for some Lipschitz constant $L>0$ which is independent of $\mu$.

Let us also define \begin{equation*}g(x,\mu)=\left(\begin{array}{c} \mu(1-x_1-x_2) \\ 0 \end{array}\right), \quad x\in\Omega.\end{equation*} Then, $g$ satisfies
\begin{equation*}||g(x,\mu)||_2=\left|\left|\left(\begin{array}{c} \mu(1-x_1-x_2) \\ 0 \end{array}\right)\right|\right|_{2}\leq \mu, \quad \forall\, x\in\Omega.\end{equation*}

We consider now the two following systems: \begin{equation*}\left(\begin{array}{c} \dot{S}_{\mathrm{SIR}}(t;x_0,\nu) \\ \dot{I}_{\mathrm{SIR}}(t;x_0,\nu) \end{array}\right)=f_{\mathrm{SIR}}(x_{\mathrm{SIR}}(t;x_0,\nu), \nu)\end{equation*} and \begin{equation*}\left(\begin{array}{c} \dot{S}_{\mathrm{SIRS}}(t;x_0,\eta) \\ \dot{I}_{\mathrm{SIRS}}(t;x_0,\eta) \end{array}\right) = f_{\mathrm{SIR}}(x_{\mathrm{SIRS}}(t;x_0,\eta), \nu)+g(x_{\mathrm{SIRS}}(t;x_0,\eta),\mu).\end{equation*}
Then, we are under the conditions of Theorem \ref{KhalilThm}. Let us make an abuse of notation and let $x_{\mathrm{SIR}}(t)=x_{\mathrm{SIR}}(t;x_0,\nu)$ and $x_{\mathrm{SIRS}}(t)=x_{\mathrm{SIRS}}(t;x_0,\eta)$. It is fulfilled that \begin{equation*}||x_{\mathrm{SIR}}(t)-x_{\mathrm{SIRS}}(t)||_2\leq \dfrac{\mu}{L}[\exp(L(t-t_0))-1], \quad \forall\, t\geq t_0.\end{equation*} This is, for any $\varepsilon>0$, $\bar{t}> t_0$, there exists some $\mu>0$ small enough such that \begin{equation}\label{tbar}||x_{\mathrm{SIR}}(t)-x_{\mathrm{SIRS}}(t)||_2\leq\dfrac{\mu}{L}[\exp(L(t-t_0))-1]<\varepsilon, \quad \forall\, t\in[t_0,\bar{t}].\end{equation}

To illustrate this, we can observe in Figure \ref{slowfast} (Left) how both solutions are hardly distinguishable when considering $\beta=2.5$, $\gamma=1$, $\mu=0.001$ and $x_{\mathrm{SIR}}(0)=x_{\mathrm{SIRS}}(0)=(0.9,0.1)^{\mathrm{T}}$. Moreover, the infectious compartment $I_{\mathrm{SIRS}}$ presents a \textit{slow-fast} behavior for the SIRS model when near to the invariant manifold $I_{\mathrm{SIRS}}\equiv 0$; we observe in Figure \ref{slowfast} (Right, Bottom) how it takes a lot of time to move away from the manifold and then approaches it very fast, and hence remains most of the time very close to the solution of $I_{SIR}$ for the SIR model. Hence, it is reasonable that $\mu=0$ is hardly distinguishable from small values of $\mu$ knowing only the infectious compartment. Moreover, Figure \ref{slowfast} illustrates that this difficulty may not only be at initial times, but also for intermediate intervals of time.
\begin{figure}[h!]
    \centering
    \begin{subfigure}{0.48\textwidth}
    \centering
        \includegraphics[width=\textwidth]{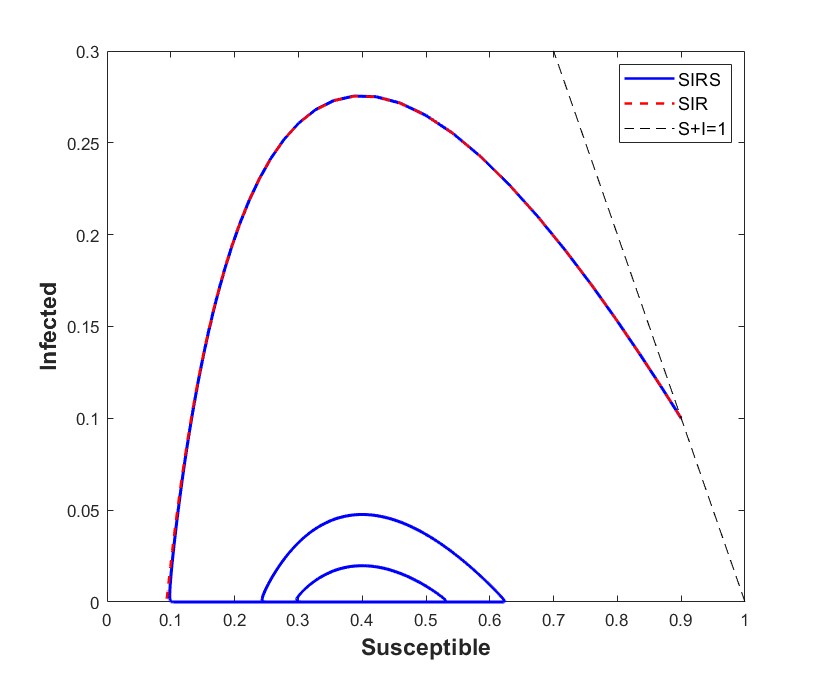}
    \end{subfigure}
    \begin{subfigure}{0.48\textwidth}
    \centering
        \includegraphics[width=\textwidth]{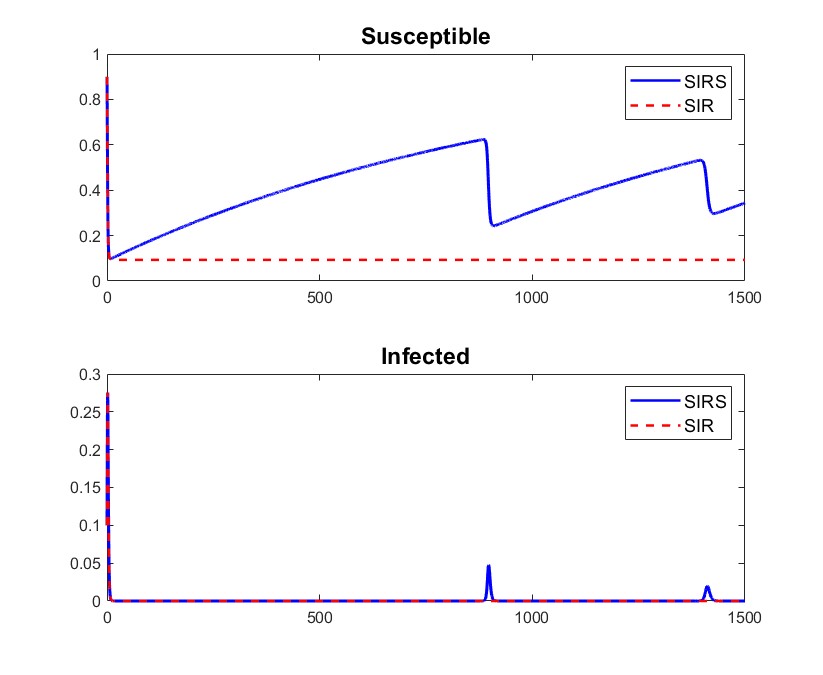}
    \end{subfigure}
    \caption{Comparison between the solutions of an SIR and an SIRS models with same initial condition $(0.9,0.1)$ and parameters $\beta=2.5$ and $\gamma=1$, $\mu=0.001$ for the SIRS model. (Left) Phase plane of both solutions, the SIR solution (dashed red) and the SIRS solution (continuous blue), along with the limiting line $S+I=1$ (dashed black). (Right) Comparison of (Top) the solution for the susceptible compartment $S$ of the SIR model (dashed red) against the one of the SIRS model (continuous blue) and (Bottom) the solution of the infectious compartment $I$ of the SIR model (dashed red) against the one of the SIRS model (continuous blue).}
    \label{slowfast}
\end{figure}

\subsubsection{Distinguishing between SIR and SIRS models}\label{distinguish}

Given the previous study, it is intuitive to think that, in a practical viewpoint, it may be difficult distinguishing an SIR model from an SIRS model with very small $\mu$ if we look at short times. However, it is theoretically possible. We present here two approaches to do this.

\begin{itemize}
\item \textbf{Approach 1:} One can check analogously to the SIRS case that $\{y,y^2,\dot{y}, y\dot{y}\}$ are linearly independent for the SIR case in any $[a,b)\subset\mathcal{I}$ in its positively invariant set $\Tilde{\Omega}=\{(\xi_1,\xi_2)^{\mathrm{T}}\in(0,1)^2 \ : \ \xi_1+\xi_2\leq 1\}$. Then, given some observations $y_{(x_0,\theta_0)}$, we can perform Step 3 of Algorithm \ref{procedure2}: we know there exist different time instants $t_1,t_2,t_3,t_4\in[0,\bar{t}]$, where $\bar{t}$ is the one in \eqref{tbar}, such that there exists a unique solution $\sigma$ fulfilling \begin{equation*}\left(\begin{array}{cccc}
    y(t_1) & y^2(t_1) & \dot{y}(t_1) & y\dot{y}(t_1) \\ \\
    y(t_2) & y^2(t_2) & \dot{y}(t_2) & y\dot{y}(t_2) \\ \\
    y(t_3) & y^2(t_3) & \dot{y}(t_3) & y\dot{y}(t_3) \\ \\
    y(t_4) & y^2(t_4) & \dot{y}(t_4) & y\dot{y}(t_4)
\end{array}\right)\left(\begin{array}{c}
     \sigma_1 \\
     \sigma_2 \\
     \sigma_3 \\
     \sigma_4
\end{array}\right)=\left(\begin{array}{c}
     \dfrac{\dot{y}^2(t_1)}{y(t_1)}-\ddot{y}(t_1) \\
     \dfrac{\dot{y}^2(t_2)}{y(t_2)}-\ddot{y}(t_2)  \\
     \dfrac{\dot{y}^2(t_3)}{y(t_3)}-\ddot{y}(t_3)  \\
     \dfrac{\dot{y}^2(t_4)}{y(t_4)}-\ddot{y}(t_4) 
\end{array}\right).\end{equation*} Given $\sigma$, we obtain the following conclusions:
\begin{itemize}
    \item if $\sigma_1=\sigma_3=0$, we confirm that our observations correspond to an SIR model;
    \item if $\sigma_3\neq 0$, then we are at an SIRS model;
    \item if $\sigma_1\neq 0$ and $\sigma_3=0$, we conclude our observations do not match any of the two models.
\end{itemize}

\item \textbf{Approach 2:} Another way to determine if there is loss of immunity or not consists in considering the previous equality \eqref{ddotSIR}, i.e., \begin{equation*}\ddot{y}-\dfrac{\dot{y}}{y}=-\dfrac{\beta\gamma}{k}y^2-\dfrac{\beta}{k}y\dot{y}, \quad \mbox{which may be rewritten as}\quad \dfrac{\mathrm{d}}{\mathrm{d} t}{\dfrac{\dot{y}}{y}}+\dfrac{\beta\gamma}{k}y+\dfrac{\beta}{k}\dot{y}=0,\end{equation*}i.e., $\left\{\dfrac{\mathrm{d}}{\mathrm{d} t}{\dfrac{\dot{y}}{y}},y,\dot{y}\right\}$ are linearly dependent for the SIR model whenever $I_{\mathrm{SIR}}\not\equiv 0$. However, this is not true for the SIRS model. Indeed, consider System \eqref{SIRSkI}. Let $a_1,a_2,a_3\in\mathbb{R}$ such that \begin{equation*}a_1\dfrac{\mathrm{d}}{\mathrm{d} t}{\dfrac{\dot{y}}{y}}+a_2y+a_3\dot{y}=0.\end{equation*}This is equivalent to \begin{equation*}a_1\beta (-\beta SI+\mu(1-S-I))+a_2kI+a_3kI(\beta S-\gamma)=0,\end{equation*}which again reduces to study the linear independence of $1,S,I,SI$, which we already know that are linearly independent in any $[a,b)\subset\mathcal{I}$ in each set $\Omega_{\theta}$, $\theta\in\Theta$, defined in Section \ref{obsnidentsubsec}. Therefore, given $y$, $\dot{y}$ and $\ddot{y}$, studying the linear dependence of $y,\dot{y},\dfrac{\mathrm{d}}{\mathrm{d} t}\dfrac{\dot{y}}{y}$ may help us determine the model.
\end{itemize}

\section{Some other applications}\label{examples}

In this section, we present a series of additional examples to demonstrate the broader applicability of our approach to other epidemiological models. In each example, we illustrate how joined observability-identifiability can be established using our proposed framework, under different modeling assumptions and data limitations. For clarity, we systematically omit the equation of the Recovered compartment in cases where the total population is constant, as it does not affect the theoretical results. To keep the focus on the theoretical aspects, detailed computational steps are omitted but can be derived straightforwardly following the procedures outlined in previous sections.

\subsection*{The SIR model with demography}

In \cite{Structural}, the authors show that the following model is neither identifiable nor observable: 
\begin{equation*}\left\{\begin{array}{ccl}
    \dot{S} & = & \delta N-\beta\dfrac{SI}{N}-\delta S, \\[0.6em]
    \dot{I} & = & \beta \dfrac{SI}{N}-(\gamma+\delta)I, \\ [0.6em]
    y &=& kI,
\end{array}\right. \quad \forall\, t\geq 0,\end{equation*}
where $\delta>0$ denotes the birth and death rates, assumed equal, $N$ is the total population (constant), and $\beta$, $\gamma$, and $k$ represent the same parameters as in the previous models.

However, this model becomes jointly observable-identifiable when the total population $N$ is known, allowing us to normalize the population so that $S$ and $I$ now represent fractions of the total population:
\begin{equation}\left\{\begin{array}{ccl}\label{dSIRSkI}
    \dot{S} & = & \delta -\beta SI-\delta S, \\
    \dot{I} & = & \beta SI-(\gamma+\delta)I, \\
    y &=& kI,
\end{array}\right. \quad \forall\, t\geq 0,\end{equation}
for some initial condition $(S(0),I(0))^{\mathrm{T}}\in[0,1]^2$.

For this normalized system, we set $\Omega=\{(\xi_1,\xi_2)^{\mathrm{T}}\in[0,1]^2 \, : \, \xi_1+\xi_2\leq 1\}$ and $\Theta:=(0,1]\times (0,+\infty)^3$ for $\theta=(k,\beta,\gamma,\delta)$. Performing computations analogous to those in Section \ref{obsnidentsubsec} during Step 1 and Step 2, we obtain the following relations:
\begin{equation*}y = kI, \quad \dot{y} = (\beta S-(\gamma+\delta))y\quad \implies \quad S = \dfrac{1}{\beta}\left(\dfrac{\dot{y}}{y}+\gamma+\delta\right), \quad I = \dfrac{y}{k},\end{equation*}
assuming $y\not\equiv0$, and
\begin{equation*}\ddot{y}-\dfrac{\dot{y}^2}{y}=y\delta(\beta-\delta-\gamma)-y^2\dfrac{\beta}{k}(\gamma+\delta)-\dot{y}\delta-y\dot{y}\dfrac{\beta}{k}.\end{equation*}

Then, we consider $d_1=1$, $d_1'=2$, $p=1$, $q_1=q=4$, 
\begin{equation*}g_0(y,\dot{y},\ddot{y})=\ddot{y}-\dfrac{\dot{y}^2}{y},\ g_1(y,\dot{y},\ddot{y})=y,\ g_2(y,\dot{y},\ddot{y})=-y^2,\ g_3(y,\dot{y},\ddot{y})=-\dot{y},\ g_4(y,\dot{y},\ddot{y})=-y\dot{y}\end{equation*}
and 
\begin{equation*}r(\theta)=\left(\delta(\beta-\gamma-\delta),\dfrac{\beta}{k}(\gamma+\delta),\delta,\dfrac{\beta}{k}\right).\end{equation*}

With these expressions, we confirm that the assumptions in Step 1 and Step 2 of Algorithm \ref{procedure2} are satisfied, taking parameters in $\Theta$ and initial conditions {consistent with the family $\{\Omega_{\theta}\}_{\theta\in\Theta}$}, with $\Omega_{\theta}:=\{(\xi_1,\xi_2)\in[0,1)\times(0,1]\ : \ \xi_1+\xi_2\leq 1\}\setminus\{P_{\mathrm{e}}(\theta)\}$, for each $\theta\in\Theta$, where $P_{\mathrm{e}}(\theta)$ denotes the equilibrium point associated with $\theta$, and $\Omega_{\theta}$ is positively invariant with respect to the ODE system of \eqref{dSIRSkI} when we consider $\theta$ as parameter vector.

Therefore, we conclude that System \eqref{dSIRSkI} is jointly observable-identifiable on {$\Gamma_{\Theta}$} in any $[a,b)\subset\mathcal{I}$.

\subsection*{The SIRV model}

We now analyze a Susceptible-Infectious-Recovered-Vaccinated (SIRV) model, where individuals gain permanent immunity after recovery, as in the SIR model. Additionally, we assume the existence of a perfect vaccine that provides permanent immunity but is only effective for susceptible individuals. Importantly, this vaccine has no effect on infectious or recovered individuals. However, since we cannot distinguish susceptible individuals from infectious or recovered ones, all compartments are vaccinated indiscriminately. The observable data are the rate of vaccinated individuals (including both effective and ineffective vaccinations). Hence, we consider the following system, where $S$, $I$, and $V$ are fractions of the total population, assumed to be known:
\begin{equation}\label{SIRV}\left\{
    \begin{array}{ccl}
        \dot{S} & = & -\beta SI-\nu S, \\
        \dot{I} & = & \beta SI - \gamma I, \\
        \dot{V} & = & \nu S, \\
        y & = & \nu (1-V),
    \end{array}\right. \quad \forall\, t\geq 0,
\end{equation}
for some initial condition $(S(0),I(0),V(0))^{\mathrm{T}}\in[0,1]^3$. We assume that the initial condition and all parameters $\beta$, $\gamma$, and $\nu$ are unknown, where $\beta$ and $\gamma$ are as defined in previous models, and $\nu>0$ (days$^{-1}$) represents the vaccination rate.

For this system, we define $\Omega=\{(\xi_1,\xi_2,\xi_3)^{\mathrm{T}}\in[0,1]^3 \ : \ \xi_1+\xi_2+\xi_3\leq1\}$ and $\Theta:=(0,+\infty)^3$ for $\theta=(\beta,\gamma,\nu)$. After some computations, we obtain the following relations:
\begin{equation*}
y=\nu(1-V), \quad
\dot{y} = -\nu^2 S, \quad
\ddot{y} = -\dot{y}(\beta I + \nu)\quad
\implies\quad S=-\dfrac{\dot{y}}{\nu^2},\quad I=-\dfrac{1}{\beta}\left(\dfrac{\ddot{y}}{\dot{y}}+\nu\right),\quad V=1-\dfrac{y}{\nu},\end{equation*}
assuming $\dot{y}\not\equiv0$, and
\begin{equation*}y^{(3)}-\dfrac{\ddot{y}^2}{\dot{y}}=-\dot{y}\nu\gamma-\dot{y}^2\dfrac{\beta}{\nu}-\ddot{y}\gamma-\dot{y}\ddot{y}\dfrac{\beta}{\nu^2}.\end{equation*}

Then, we consider $d_1=2$, $d_1'=3$, $p=1$, $q_1=q=4$, 
\begin{equation*}g_0(y,\dot{y},\ddot{y},y^{(3)})=y^{(3)}-\dfrac{\ddot{y}^2}{\dot{y}},\quad g_1(y,\dot{y},\ddot{y},y^{(3)})=-\dot{y},\quad g_2(y,\dot{y},\ddot{y},y^{(3)})=-\dot{y}^2,\end{equation*} 
\begin{equation*}g_3(y,\dot{y},\ddot{y},y^{(3)})=-\ddot{y},\quad g_4(y,\dot{y},\ddot{y},y^{(3)})=-\dot{y}\ddot{y},\end{equation*} 
and 
\begin{equation*}r(\theta)=\left(\nu\gamma,\dfrac{\beta}{\nu},\gamma,\dfrac{\beta}{\nu^2}\right).\end{equation*}

Following the procedures established in earlier sections, we confirm that the assumptions of Step 1 and Step 2 in Algorithm \ref{procedure2} are met, with parameters in $\Theta$ and initial conditions in the revised set $\Omega=\{(\xi_1,\xi_2,\xi_3)\in(0,1]\times(0,1)\times[0,1)\ : \ \xi_1+\xi_2+\xi_3\leq 1\}$, where $\Omega$ does not depend on $\theta\in\Theta$ and is positively invariant under the ODE system defined by \eqref{SIRV}.

Thus, we conclude that System \eqref{SIRV} is jointly observable-identifiable on $\Omega\times\Theta$ in any $[a,b)\subset\mathcal{I}$.

\subsection*{The SIR model (with a different output)}

As discussed in Section \ref{particularSIR}, the SIR model observed through a fraction of infected individuals (i.e., System \eqref{SIRkI}) is observable and identifiable, but not jointly observable-identifiable. Here, we consider the same SIR model but with a different observation: the instantaneous incidence rate. Specifically, we examine the following system:
\begin{equation}\label{SIRbSI}
    \left\{\begin{array}{lcl}
        \dot{S} & = & -\beta SI, \\
        \dot{I} & = & \beta SI-\gamma I, \\
        y & = & \beta SI,
    \end{array}\right.
\end{equation}
for some initial condition $(S(0),I(0))\in\Omega=\{(\xi_1,\xi_2)^{\mathrm{T}}\in[0,1]^2 \ : \ \xi_1+\xi_2\leq 1\}$, which is positively invariant with respect to the system of ODEs given in \eqref{SIRbSI}. In this scenario, we assume that both $\beta$ and $\gamma$ are unknown parameters, defining $\Theta:=(0,\infty)^2$ for $\theta=(\beta,\gamma)$. Following an approach similar to previous cases, we calculate the first derivative of $y$ to express $S$ and $I$ in terms of $y$ and $\dot{y}$:
\begin{equation*}\left\{\begin{array}{lcl}
    y & = & \beta SI, \\
    \dot{y} & = & -y\beta I+y\beta S-\gamma y.
\end{array}\right.\end{equation*}

Notice that this system is nonlinear in $S$ and $I$, making it challenging to directly apply Step 1 of Algorithm \ref{procedure2}. To address this, let us differentiate once more: 
\begin{equation*}\ddot{y}=-\dot{y}\beta I-2y^2\beta+y\beta\gamma I+\dot{y}\beta S-\gamma \dot{y},\end{equation*}
which results in a system that is linear in $S$ and $I$. Therefore, we obtain:
\begin{equation*}\left\{ \begin{array}{lcl}
    \dot{y} & = & (-\beta I+\beta S-\gamma) y, \\
    \ddot{y} & = & (-\beta I+\beta S-\gamma) \dot{y}-2y^2\beta+y\beta\gamma I,
\end{array}\right. \implies \left\{\begin{array}{lcl}
    S & = & \dfrac{1}{\beta}\dfrac{\dot{y}}{y}+\dfrac{1}{\beta\gamma}\left(\dfrac{\ddot{y}}{y}-\dfrac{\dot{y}^2}{y^2}\right)+\dfrac{2}{\gamma}y+\dfrac{\gamma}{\beta}, \\[1em]
    I & = & \dfrac{1}{\beta\gamma}\left(\dfrac{\ddot{y}}{y}-\dfrac{\dot{y}^2}{y^2}\right)+\dfrac{2}{\gamma}y.
\end{array}\right.\end{equation*}

Assuming $y\not\equiv0$, we avoid further differentiation by substituting the expressions of $S$ and $I$ into $y=\beta SI$, yielding the following equation after simplification:
\begin{equation*}-\left(\dfrac{\mathrm{d}}{\mathrm{d} t}\dfrac{\dot{y}}{y}\right)^2=\gamma\dfrac{\dot{y}}{y}\dfrac{\mathrm{d}}{\mathrm{d} t}\dfrac{\dot{y}}{y}+2\beta\gamma\dot{y}+\beta\gamma^2y+4\beta y\dfrac{\mathrm{d}}{\mathrm{d} t}\dfrac{\dot{y}}{y}+4\beta^2y^2+\gamma^2\dfrac{\mathrm{d}}{\mathrm{d} t}\dfrac{\dot{y}}{y}.\end{equation*}

Then, we set $d_1=d_1'=2$, $p=1$, $q_1=q=6$, 
\begin{equation*}g_0(y,\dot{y},\ddot{y})=-\left(\dfrac{\mathrm{d}}{\mathrm{d} t}\dfrac{\dot{y}}{y}\right)^2, \quad g_1(y,\dot{y},\ddot{y})=\dfrac{\dot{y}}{y}\dfrac{\mathrm{d}}{\mathrm{d} t}\dfrac{\dot{y}}{y}, \quad g_2(y,\dot{y},\ddot{y})=\dot{y}, \quad g_3(y,\dot{y},\ddot{y})=y,\end{equation*} 
\begin{equation*}g_4(y,\dot{y},\ddot{y})=y\dfrac{\mathrm{d}}{\mathrm{d} t}\dfrac{\dot{y}}{y}, \quad g_5(y,\dot{y},\ddot{y})=y^2, \quad g_6(y,\dot{y},\ddot{y})=\dfrac{\mathrm{d}}{\mathrm{d} t}\dfrac{\dot{y}}{y},\end{equation*} 
and 
\begin{equation*}r(\theta)=\left(\gamma,2\beta\gamma,\beta\gamma^2,4\beta,4\beta^2,\gamma^2\right).\end{equation*}

We confirm that the assumptions of Step 1 and Step 2 in Algorithm \ref{procedure2} hold, with parameters in $\Theta$ and initial conditions in $\Omega=\{(\xi_1,\xi_2)\in(0,1)^2 \ : \ \xi_1+\xi_2\leq 1\}$, where $\Omega$ does not depend on $\theta\in\Theta$ and is positively invariant with respect to the system of ODEs of System \eqref{SIRbSI}.

Consequently, we conclude that System \eqref{SIRbSI} is jointly observable-identifiable on $\Omega\times\Theta$ in any $[a,b)\subset\mathcal{I}$.

\subsection*{The SIV model with demography and two outputs}

Here, we consider a Susceptible-Infectious-Vaccinated (SIV) model with demographic dynamics. We assume a perfect vaccine that is effective only for susceptible individuals, though infectious individuals are also vaccinated indiscriminately, as in System \eqref{SIRV}. In this case, however, we model a permanent infection, which can serve as a simplified representation of long-term infections, such as those associated with certain sexually transmitted diseases (e.g., \textit{human papillomavirus} (HPV)). We assume two observable quantities: the rate of all vaccinated individuals (both effective and ineffective) and the rate of natural deaths. These deaths are attributed to natural population dynamics and are assumed to be routinely monitored by authorities. Hence, we consider the following system, where $S$, $I$, and $V$ represent the fraction of susceptible, infectious, and vaccinated individuals, respectively:
\begin{equation}\label{SIVd}
    \left\{\begin{array}{lcl}
        \dot{S} & = & A-\beta SI - (\nu+\delta)S, \\
        \dot{I} & = & \beta SI-\delta I, \\
        \dot{V} & = & \nu S-\delta V, \\
        y & = & \left(\begin{array}{c}
             y_1 \\
             y_2
        \end{array}\right) = \left(\begin{array}{c}
             \nu(S+I) \\
             \delta(S+I+V)
        \end{array}\right),
    \end{array}\right. \quad \forall\, t\geq 0,
\end{equation}
for some initial condition $(S(0),I(0),V(0))\in[0,\infty)^3$. All the initial condition and the parameters $A$, $\beta$, $\delta$, and $\nu$ are assumed to be unknown. Here, $A>0$ represents a constant recruitment rate, $\delta>0$ is the death rate, and $\beta$ and $\nu$ have the same definitions as in System \eqref{SIRV}.

The natural positively invariant set for this system is $\Omega=\{(\xi_1,\xi_2,\xi_3)^{\mathrm{T}}\in\left[0,A/\delta\right]^3 \ : \ \xi_1+\xi_2+\xi_3\leq A/\delta\}$, and $\Theta := (0,\infty)^4$, for $\theta=(A,\beta,\delta,\nu)$.

We first observe that this system is not observable (and hence not jointly observable-identifiable) if only $y_1$ is used as the observation, similar to System \eqref{SIRV}. Since $y_1$ depends solely on $S$ and $I$, we cannot uniquely determine $V_0$. Therefore, an additional observation is required. By including the rate of natural deaths, $y_2=\delta(S+I+V)$, we leverage data that should be routinely available, making it reasonable to assume accessibility to this information.

Similarly, relying on $y_2$ alone is insufficient, as $\dot{y}_2=A\delta-\delta y_2$ yields no further parameter information through differentiation, unlike previous models.

After some computations, differentiating $y_1$ twice and $y_2$ once, we obtain the following expressions:
\begin{equation*}\left\{\begin{array}{lcl}
    y_1 &=& \nu(S+I), \\
    \dot{y}_1 &=& \nu A-\nu^2 S-\delta y_1, \\
    y_2 &=& \delta(S+I+V)
\end{array}\right.\implies
\left\{\begin{array}{lcl}
    S &=& \dfrac{A}{\nu}-\dfrac{\delta}{\nu^2}y_1-\dfrac{\dot{y}_1}{\nu^2}, \\[1em]
    I &=& \left(\dfrac{\nu+\delta}{\nu^2}\right)y_1-\dfrac{A}{\nu}+\dfrac{\dot{y}_1}{\nu^2}, \\[1em]
    V &=& \dfrac{y_2}{\delta}-\dfrac{y_1}{\nu},
\end{array}\right.\end{equation*}
and 
\begin{equation*}\left\{\begin{array}{lcl}
    \dot{y}_1^2 &=& A\nu^2\dfrac{\delta\nu-A\beta}{\beta}+\nu\left(A\nu+2A\delta-\dfrac{\delta\nu^2+\delta^2\nu}{\beta}\right)y_1+\nu\left(2A-\dfrac{\nu^2+2\delta\nu}{\beta}\right)\dot{y}_1 \\
    &&-\delta(\nu+\delta)y_1^2-(\nu+2\delta)y_1\dot{y}_1-\dfrac{\nu^2}{\beta}\ddot{y}_1, \\
    \dot{y}_2 &=& A\delta-\delta y_2.
\end{array}\right.\end{equation*}

Focusing on the equation $\dot{y}_2=A\delta-\delta y_2$, we observe that $\{1, y_2\}$, when $y_2$ is non-constant, allows us to determine $A$ and $\delta$ by solving the system:
\begin{equation*}\begin{pmatrix}
    1 & y_2(t_1) \\
    1 & y_2(t_2)
\end{pmatrix}\begin{pmatrix}
    \sigma_1 \\ \sigma_2
\end{pmatrix}=\begin{pmatrix}
    \dot{y}_2(t_1) \\
    \dot{y}_2(t_2)
\end{pmatrix},\end{equation*}
such that $\delta = -\sigma_2$ and $A=-\sigma_1/\sigma_2$. Substituting, we can rewrite:
\begin{equation*}\left\{\begin{array}{rcl}
    \dot{y}_1^2+\sigma_2^2y_1^2-2y_1\dot{y}_1\sigma_2 &=& A\nu^2\dfrac{\delta\nu-A\beta}{\beta}+\nu\left(A\nu+2A\delta-\dfrac{\delta\nu^2+\delta^2\nu}{\beta}\right)y_1 \\[1em]
    &&+\nu\left(2A-\dfrac{\nu^2+2\delta\nu}{\beta}\right)\dot{y}_1-\nu y_1\left(-\sigma_2y_1+\dot{y}_1\right)-\dfrac{\nu^2}{\beta}\ddot{y}_1, \\
    \dot{y}_2 &=& A\delta-\delta y_2.
\end{array}\right.\end{equation*}

This reformulation ensures the linear independence condition in Step 2 of Algorithm \ref{procedure2}. Then, we consider $d_1=1$, $d_2=0$, $d_1'=2$, $d_2'=1$, $p=2$, $q_1=5$, $q_2=2$,
\begin{eqnarray*}g_{1,0}(y_1,\dot{y}_1,\ddot{y}_1,y_2,\dot{y}_2)&=&\dot{y}_1^2+\sigma_2y_1^2-2y_1\dot{y}_1\sigma_2, \quad g_{1,1}(y_1,\dot{y}_1,\ddot{y}_1,y_2,\dot{y}_2)=1, \\ g_{1,2}(y_1,\dot{y}_1,\ddot{y}_1,y_2,\dot{y}_2)&=&y_1, \quad g_{1,3}(y_1,\dot{y}_1,\ddot{y}_1,y_2,\dot{y}_2)=\dot{y}_1, \\
g_{1,4}(y_1,\dot{y}_1,\ddot{y}_1,y_2,\dot{y}_2)&=&-y_1\left(-\sigma_2y_1+\dot{y}_1\right), \quad g_{1,5}(y_1,\dot{y}_1,\ddot{y}_1,y_2,\dot{y}_2)=-\ddot{y}_1, \\
g_{2,0}(y_1,\dot{y}_1,\ddot{y}_1,y_2,\dot{y}_2)&=&\dot{y}_2,\quad g_{2,1}(y_1,\dot{y}_1,\ddot{y}_1,y_2,\dot{y}_2)=1,\quad g_{2,2}(y_1,\dot{y}_1,\ddot{y}_1,y_2,\dot{y}_2)=-y_2,\end{eqnarray*}and 
\begin{equation*}r(\theta_0)=\left(A\nu^2\dfrac{\delta\nu-A\beta}{\beta},\nu\left(A\nu+2A\delta-\dfrac{\delta\nu^2+\delta^2\nu}{\beta}\right),\nu\left(2A-\dfrac{\nu^2+2\delta\nu}{\beta}\right),\nu,\dfrac{\nu^2}{\beta},A\delta,\delta\right).\end{equation*}

Given this, the assumptions in Step 1 and Step 2 of Algorithm \ref{procedure2} hold for parameters in $\Theta$ and initial conditions {consistent with the family $\{\Omega_{\theta}\}_{\theta\in\Theta}$}, such that, for $\theta\in\Theta$, $\Omega_{\theta}:=\{(\xi_1,\xi_2,\xi_3)\in[0,A/\delta)\times(0,A/\delta)\times[0,A/\delta)\ : \ \xi_1+\xi_2+\xi_3< A/\delta\}\setminus\{P_{\mathrm{e}}(\theta)\}$, where $P_{\mathrm{e}}(\theta)$ is the equilibrium point of the system and $\Omega_{\theta}$ is positively invariant under the ODEs of System \eqref{SIVd} when considering $\theta$ as the parameter vector.

Therefore, we conclude that System \eqref{SIVd} is jointly observable-identifiable on {$\Gamma_{\Theta}$} in any $[a,b)\subset\mathcal{I}$.

\section{Numerical illustration}\label{numerical}

In this section, we present numerical experiments designed to illustrate the application of the theoretical framework developed in Sections \ref{gen_model} and \ref{obsnident} for observability, identifiability and joined observability-identifiability in epidemiological models. This numerical analysis also aims to implement the theoretical methods in some cases and demonstrate their practical feasibility under varying conditions. Specifically, we analyze two examples: one based on an SIRS model and another on an SIR model, each chosen to exhibit distinct epidemiological dynamics:
\begin{itemize}
    \item \textbf{Case 1:} We consider an SIRS model in which the endemic equilibrium is globally asymptotically stable (except for the invariant manifold $I\equiv 0$, as it is stable for the disease-free equilibrium). The solution oscillates when converging to the endemic equilibrium, as in the example presented in Figure \ref{slowfast}. We set the following parameters and initial conditions: $k_0=0.3,\ \beta_0=0.25, \ \gamma_0 =0.1, \ \mu_0=0.05, \ S_0=0.9,$ and $I_0=0.1$. Thus, the basic reproduction number is $\mathcal{R}_0=2.5$.
    \item \textbf{Case 2:} We consider an SIR model that exhibits an epidemic peak. The chosen parameters and initial conditions are $k_0=0.3, \ \beta_0=0.25, \ \gamma_0 =0.1, \ S_0=0.9,$ and $I_0=0.1$. Here as well, $\mathcal{R}_0=2.5$.
\end{itemize}

To simulate these cases, we approximate each solution using a fourth-order Runge-Kutta algorithm and synthetically generate observations, $y=k_0I$, at each time-step. {In this initial study, we assume noise-free data to establish a clear baseline for comparison.} We then analyze two scenarios:
\begin{itemize}
    \item \textbf{Scenario 1:} We assume that continuous observations of $y$ are available over the interval $[0,T_{\max}]$, with exact knowledge or computability of its derivatives (enabled by the analyticity of the system). We apply the recovery procedure outlined in Algorithm \ref{procedure3} to determine the original parameters and initial condition, utilizing the full interval $[0,T_{\max}]$.
    \item \textbf{Scenario 2:} In realistic situations, data are typically discrete, often recorded at daily intervals by public health authorities (e.g., \cite{WHOdailyEbola} and \cite{WHOdailyCOVID}). To mimic this, we extract simulated data once per day (at the same time each day), referred to here as \textit{daily data}, and apply the Ordinary Least Squares method to estimate the unknown parameters.
\end{itemize}

{Scenario 1 provides a controlled environment for testing recovery procedures and Scenario 2 reflects more realistic, discrete data conditions.} For practical purposes, we will examine Case 1 under Scenario 1 and Case 2 under Scenario 2. Preliminary results indicated similar outcomes when interchanging cases between scenarios, allowing us to streamline the analysis. In Scenario 1, we focus on the procedure in Algorithm \ref{procedure3}, omitting that of Algorithm \ref{procedure2} due to the former's superior performance.

To integrate numerically the ODE systems, we utilize a fourth-order Runge-Kutta method with a time-step of $h=2^{-5}$ days, extending up to a maximum simulation time of $T_{\max}=5$ days. {Although the maximum time may seem low, it was sufficient to capture the dynamics required for convergence in Scenario 2.} Furthermore, we employ the extended SIRS model from Section \ref{particularSIR}, assuming no prior knowledge regarding whether $\mu_0 \neq 0$ or $\mu_0 = 0$.

\subsection{Scenario 1: Linear systems for continuous observations}

In this section, we perform numerical tests for Case 1, assuming continuous observation of $y=k_0I$ over $[0,T_{\max}]$, i.e., we observe $y$ at every time point within this interval, and assume that we know or can compute its exact successive derivatives.

To implement the numerical tests, we first need to construct the necessary data. From an implementation perspective, we obtain the values of $y$ and its derivatives at each time-step of the numerical scheme, specifically at each point in the time vector $t_{\mathrm{vec}}=(t_i)_{i\in\{0,\dots,N\}}$, where $t_i=ih$ and $N=T_{\max}/h\in\mathbb{N}$. The data generation steps are as follows:
\begin{itemize}
    \item[1.] We approximate the solutions $S$ and $I$ of the SIR or SIRS system with a small time-step $h$ using the fourth-order Runge-Kutta scheme, and define our observations as $y=k_0I$.
    \item[2.] We approximate the first derivative of $y$ by $\dot{y}=k_0\dot{I}=k_0I(\beta_0 S-\gamma_0)$.
    \item[3.] For higher-order derivatives, we use the following linear equation in terms of parameters $\sigma_i$, $i\in\{1,2,3,4\}$:
    \begin{equation*}
    \ddot{y}=\dfrac{\dot{y}^2}{y}-\sigma_1y-\sigma_2y^2-\sigma_3\dot{y}-\sigma_4y\dot{y},
    \end{equation*}
    where $\sigma_1=\mu_0(\gamma_0-\beta_0)$, $\sigma_2=\beta_0(\gamma_0+\mu_0)/k_0$, $\sigma_3=\mu_0$, and $\sigma_4=\beta_0/k_0$. Using this relationship, we iteratively approximate higher-order derivatives of $y$ through lower-order derivatives.
\end{itemize}

With these generated data, we apply the method in Algorithm \ref{procedure3} to estimate the system parameters and initial condition. This method requires selecting a time $\tilde{t}\in t_{\mathrm{vec}}$ such that the following system has a unique solution (see \eqref{step5proc3}-\eqref{step6proc3}):
\begin{equation}\label{methodb}
\left(\begin{array}{cccc}
    \dfrac{\mathrm{d}^k y}{\mathrm{d} t^k} & \dfrac{\mathrm{d}^k y^2}{\mathrm{d} t^k} & \dfrac{\mathrm{d}^k \dot{y}}{\mathrm{d} t^k} & \dfrac{\mathrm{d}^k y\dot{y}}{\mathrm{d} t^k}
\end{array}\right)_{\begin{subarray}{l} t=\tilde{t} \\ k=0,1,2,3\end{subarray}}\sigma^{\mathrm{T}} =\left(\begin{array}{c}
    \dfrac{\mathrm{d}^k}{\mathrm{d} t^k}\left(\dfrac{\dot{y}^2}{y}-\ddot{y}\right)
\end{array}\right)_{\begin{subarray}{l} t=\tilde{t} \\ k=0,1,2,3\end{subarray}}.
\end{equation}

Instead of directly estimating $(k_0,\beta_0,\gamma_0,\mu_0)$, we focus on estimating
\begin{equation*}
\sigma=(\sigma_1,\sigma_2,\sigma_3,\sigma_4)=\left(\mu_0(\gamma_0-\beta_0),\dfrac{\beta_0}{k_0}(\gamma_0+\mu_0),\mu_0,\dfrac{\beta_0}{k_0}\right).
\end{equation*}
The initial conditions $S_0$ and $I_0$ are then computed as
\begin{equation*}
S_0=\dfrac{1}{\beta_0}\left(\dfrac{\dot{y}(0)}{y(0)}+\gamma_0\right) \quad \text{and} \quad I_0=\dfrac{y(0)}{k_0}.
\end{equation*}
Knowing the observation at time 0 avoids the computational challenges of backward integration.

Assuming known values of $y(0)=0.03$ and $\dot{y}(0)=0.00375$, we aim to calibrate
\begin{equation*}
(\sigma_1,\sigma_2,\sigma_3,\sigma_4)=(-0.0075,\ 0.125,\ 0.05,\ 0.8\overset{\_}{3}) \quad \text{and}\quad (S_0,I_0)=(0.9,0.1).
\end{equation*}

To evaluate the method's performance, we conducted 161 experiments, corresponding to $N+1=T_{\max}/h+1=161$. For each time point in $t_{\mathrm{vec}}$, we examined:
\begin{itemize}
    \item The relative error between each computed component of $\sigma$ and its exact value. {These errors were consistently small, with a maximum of the order of $10^{-13}$, confirming the high accuracy of the method.}
    \item The determinant of the matrix in \eqref{methodb}, which should be non-zero. The values obtained were of the order of $10^{-17}$.
    \item The condition number of the matrix in \eqref{methodb} (based on the $L^2$-norm). All the obtained values range between $10^3$ and $5\times 10^5$.
    \item The computational time (in seconds). Performing all tests took approximately $10^{-2}$ seconds.
\end{itemize}

These results indicate that the method provides accurate parameter estimates despite the small determinant values and the matrix conditioning. {The computational efficiency demonstrated, completing 161 tests in minimal time, makes this approach highly suitable for real-time applications.}

Similar results were observed for Case 2, which are omitted here for brevity.

\subsection{Scenario 2: OLS method for daily observations}\label{OLS}

In this section, we analyze Case 2 under the assumption of daily, noise-free observations from the deterministic model. {Daily observations reflect realistic data collection practices in epidemiology, such as those by public health
authorities.} Given the discrete nature of the data, calculating derivatives directly can result in inaccuracy (if differentiated numerically) or bias (if interpolated prior to differentiation). In this context, we employ the Ordinary Least Squares (OLS) method to estimate the values of $k_0$, $\beta_0$, $\gamma_0$, $\mu_0$, and $S_0$. Let $\theta=(\theta_1,\theta_2,\theta_3,
\theta_4,\theta_5)$ represent the parameter vector, where we calibrate 4 parameters ($k_0$, $\beta_0$, $\gamma_0$, and $\mu_0$) and 1 initial condition ($S_0$). The goal is to find $\theta_{\mathrm{OLS}}=(k_{\mathrm{OLS}},\beta_{\mathrm{OLS}},\gamma_{\mathrm{OLS}},\mu_{\mathrm{OLS}},S_{0,\mathrm{OLS}})$ by solving 
\begin{equation*} 
\theta_{\mathrm{OLS}}=\underset{\theta\in{\Theta}}{\mathrm{argmin}}\sum_{i=0}^{T_{\max}} A\left(\theta_1 I(t_i;\theta)-y(t_i)\right)^2, \quad t_i=i,
\end{equation*} 
where $I(t;\theta)$ is the solution to the following system of ODEs:
\begin{equation*}
\left\{\begin{array}{lcl}
    \dot{S}(t;\theta) & = & -\theta_2 S(t;\theta)I(t;\theta)+\theta_4(1-S(t;\theta)-I(t;\theta)), \\[.6em]
    \dot{I}(t;\theta) & = & \theta_2 S(t;\theta)I(t;\theta)-\theta_3 I(t;\theta),
\end{array}\right.
\end{equation*}
with $(S(0;\theta),I(0;\theta))=\left(\theta_5,\dfrac{y(0)}{\theta_1}\right)$. Here, we define a new feasible set $\Theta^*=\Theta\times[0,1)$. To {enhance sensitivity to small deviations, a scaling factor $A=10^{14}$ is introduced in the objective function.}

The choice of $T_{\max}=5$ days (using daily data) was based on preliminary testing, where this time-frame was found sufficient for reliable parameter estimation despite its brevity.

For this experiment, we use the \texttt{MATLAB} function \texttt{lsqcurvefit} to solve the OLS problem, adjusting settings for increased accuracy (see \cite{lsqcurvefit}). Specifically, we set the maximum evaluations to $10^6$, iterations to $5 \times 10^5$, step tolerance to $10^{-15}$, and function tolerance to $10^{-17}$.

We impose the following bounds for the parameters, with subscripts $m$ and $M$ indicating the minimum and maximum values, respectively: $[k_m, k_M]=[\min_t\{y\},1]$, $[\beta_m,\beta_M]=[10^{-2},3]$, $[\gamma_m, \gamma_M]=[10^{-2},1]$, $[\mu_m,\mu_M]=[0,1]$, and $[S_m,S_M]=[0,1-10^{-10}]$. {The bounds for $\beta$ are informed by studies on $\mathcal{R}_0$ estimates for diseases such as smallpox, pertussis, or COVID-19} (e.g., \cite{R0smallpox}, \cite{R0pertussis}, \cite{R0delta}, \cite{R0omicron}, \cite{R0ebola}, \cite{WHOsars}).

We initialize the OLS algorithm several times with random initial conditions, denoted as I.C., drawn from a uniform distribution $\mathcal{U}(0,1)$. Each initial condition is generated as:
\begin{equation*}
\text{I.C.}=(k_m,\beta_m,\gamma_m,\mu_m,S_m)^{\mathrm{T}}+(\rho_1 (k_M-k_m), \rho_2(\beta_M-\beta_m), \rho_3(\gamma_M-\gamma_m), \rho_4(\mu_M-\mu_m), \rho_5(S_M-s_m))^{\mathrm{T}},
\end{equation*}
where $\rho_i\sim \mathcal{U}(0,1)$. For conciseness, we report results from 5 representative tests. Table \ref{0OLSv1} includes the following outcomes: Test (test number), I.C. (initial condition for \texttt{lsqcurvefit}), Abs. error $\theta_0$ (absolute error between the computed solution and the true parameter vector $\theta_0$), Obj. value (final objective function value), and Time (s) (computational time for each test in seconds).

\begin{table}[h!]
\small
\begin{center}
\begin{tabular}{ c|c|c|c|c }
 Test & I.C. & Abs. error $\theta_0$ & Obj. value & Time (s) \\
 \hline
 1 & (0.545    1.967    0.414    0.82    0.718) & (0.019 0.015 6.8e-6 4.7e-6 0.052) & 6.324e-12 & 116.517 \\
 \hdashline
 2 & (0.97    1.599    0.332    0.106   0.611) & (0.226 0.189 8.9e-12 9.3e-13 0.387) & 7.657e-19 & 1.904 \\
 \hdashline
 3 & (0.785    1.276    0.1    0.266    0.154) & (0.25 0.208 4e-11 3.6e-12 0.409) & 9.036e-18 & 1.546 \\
 \hdashline
 4 & (0.686    0.874    0.675    0.695    0.068) & (0.7 0.286 0.270 0.096 0.024) & 0.003 & 621.808 \\
 \hdashline
 5 & (0.277    0.68    0.671    0.844    0.344) & (0.7 0.287 0.271 0.097 0.024) & 0.006 & 1380.792
\end{tabular}
    \caption{Results for Scenario 2 applied to Case 2 using the OLS method.\label{0OLSv1}}
    \vspace{1em}
\begin{tabular}{ c|c|c }
Test & Approx. of $\theta_0$ & Approx. of $\tilde{\theta}_0$ \\
\hline
 1 & (0.319  0.265   0.1    4.7e-6    0.848) & (0.1   0.833   0.225) \\
 \hdashline
 2 & (0.526   0.439   0.1   9.3e-13   0.513) & (0.1    0.833   0.225) \\
 \hdashline
 3 & (0.55  0.458   0.1 3.6e-12 0.491) & (0.1   0.833   0.225) \\
 \hdashline
 4 & (1  0.536    0.370   0.096    0.924) & (0.370  0.536    0.495) \\
 \hdashline
 5 & (1  0.537   0.371   0.097 0.924) & (0.371   0.537   0.496)
\end{tabular}
    \caption{Results for Scenario 2 applied to Case 2 using the OLS method: Approximations of the parameters.\label{0OLSv1params}}
\end{center}
\end{table}

The target parameter vector is
\begin{equation*}
\theta_0=(k_0,\beta_0,\gamma_0,\mu_0,S_0)=(0.3, \ 0.25, \ 0.1, \ 0, \ 0.9),
\end{equation*}
but as noted in Section \ref{particularSIR}, we can only determine specific combinations: $\gamma_0$, $\beta_0/k_0$, $\beta S_0$, $kI_0$. Defining $\tilde{\theta}_0=(\gamma_0,\beta_0/k_0,\beta_0 S_0)$,
\begin{equation*}
\tilde{\theta}_0=(\gamma_0,\beta_0/k_0,\beta_0 S_0)=(0.1,0.8\overset{\_}{3},0.225).
\end{equation*}
Table \ref{0OLSv1params} compares approximations for $\theta_0$ and $\tilde{\theta}_0$, showing accurate convergence for $\gamma_0$ and $\mu_0$ in Tests 1, 2 and 3. Despite inaccuracies in $k_0$, $\beta_0$, and $S_0$ in these tests, the objective function values are low, indicating that the algorithm converged successfully for the key identifiable combinations. Most of the tests not presented here converged similarly.

In Tests 4 and 5, the parameters reach a similar solution, representing an SIRS model with 
\begin{equation*}
k\approx 1, \, \beta \approx 0.54, \, \gamma \approx 0.37, \, \mu\approx 0.1, \, \text{and} \, S(0)\approx 0.924.
\end{equation*}
{These results further illustrate the method’s capacity to handle challenging scenarios where early-stage data provide limited discriminatory power}. The amplified error factor of $10^{14}$ implies that the non-amplified error in these cases is of the order of $10^{-17}$, illustrating the difficulty in distinguishing between SIR and SIRS models in early stages when only a portion of infected individuals is observed, as noted in Section \ref{compSIRS}. Several tests not presented here reached similar solutions. The rest of the tests showed different results.

For Case 1, similar results were observed, with certain tests {producing parameter sets indicative of an SIR model (i.e., with $\mu\approx 0$). Nonetheless, these tests consistently preserved the identifiable combinations $\gamma$, $\beta/k$ and $\beta S_0$, resulting in undistinguishable SIR models.} These results are not presented here for brevity.

\section{Conclusions}

This work addresses the problems of observability, identifiability and joined identifiability-observability in a wide class of systems of ODEs that covers classical epidemiological models, considering idealized continuous, noise-free observations. This problem was previously proposed in \cite{Kubik}, where the basis of the approach was developed in a general context. In Section \ref{gen_model}, we extended the theoretical framework in \cite{Kubik} considering parameter-dependent sets of initial conditions, and presented Algorithms \ref{procedure1}, \ref{procedure2} and \ref{procedure3}, offering different constructive methods to compute unknown parameters and/or initial conditions based on this framework. These methods rely on solving a series of linear systems at well chosen time instants, providing a practical approach to recovering model parameters. Additionally, if observations are available at time 0, the initial condition can be determined directly; otherwise, it can be inferred by integrating backwards using the estimated parameters. Furthermore, our approach can be extended to models with piecewise constant parameters by allowing parameter updates at predetermined time instants.

One of the primary ideas underlying our theoretic framework is the establishment of the linear independence of some sets of functions, an approach that has not been much addressed, and the known results are framed in more restricted settings (e.g., rational models, or under the knowledge of data at initial time). This technique, although challenging depending on the sets of functions, has proved to be essential for deriving the joined observability-identifiability results. Moreover, this method reduces the dependency on high-order derivatives, offering simpler computations and more robust algorithms against noise compared to classical techniques such as the Hermann-Krener matrix rank condition.

In Section \ref{obsnident}, we applied this framework to an SIRS model (with $\mu>0$) when only a fraction of infected individuals is observed, under the assumption of continuous, noise-free observations, and demonstrated its theoretical joined observability-identifiability. Specifically, we showed that observing only a fraction of the infected population is sufficient to determine all model parameters ($k$, $\beta$, $\gamma$ and $\mu$) and the initial condition of both susceptible and infected individuals. Furthermore, it is demonstrated that considering parameter-dependent sets of initial conditions is crucial in some cases.

For the SIR model (which can be regarded as a limit case of the SIRS model with $\mu=0$), however, we found that it is observable and identifiable, but not jointly observable-identifiable. Specifically, if both the initial condition and the parameters are unknown, we can only determine $\gamma$, $\beta/k$, $\beta S_0$ and $kI_0$. This highlights the critical role of the SIRS model’s cyclic dynamics (loss of immunity) in achieving joined observability-identifiability. Based on these findings, we proposed methods to distinguish data from SIR and SIRS models theoretically.

In Section \ref{examples}, we presented additional epidemiological examples, using one and two observations, for which our framework can also be applied. These examples illustrate the adaptability of the proposed methodologies across different model structures and observational scenarios.

Finally, in Section \ref{numerical} we demonstrated the applicability of our theoretical results through numerical experiments, using an SIRS model and an SIR model, each under two different scenarios:
\begin{itemize}
    \item \textbf{Scenario 1:} The observations are \textit{continuous}, for which we considered the numerical solutions at each time-step and assumed perfect knowledge of the derivatives of the observations. We solved this scenario through the procedure of Algorithm \ref{procedure3}. All tests approximated successfully the exact solution of the SIRS case.
    \item \textbf{Scenario 2:} The observations in this scenario were considered to be daily reported (i.e., discrete), and we used Ordinary Least Squares for the calibration of the parameters and presented the results of 5 representative tests for the SIR case. {The results indicated that the method reliably captured most of the times the key identifiable combinations ($\gamma$, $\beta/k$, $\beta S(0)$) and $\mu$ approximately 0, as expected from the theory exposed in Section \ref{particularSIR}. However, two tests produced approximately the same parameter vectors corresponding to an SIRS model. This result illustrates the difficulty of distinguishing between the two models during early epidemic stages when only a portion of infected individuals is observed, as discussed in Section \ref{compSIRS}.}
\end{itemize}

{This study demonstrates the feasibility of the proposed methods for parameter estimation and model discrimination in epidemiological systems.} Future work should explore extensions to scenarios with more realistic features, such as noisy data (see, e.g., \cite{IdentEbola}) or bounded measurement ranges (e.g., related to \textit{interval observers}, see \cite{RapaportDochain}).

\section*{Acknowledgments}

This work was carried out with financial support from the projects PID2019-106337GB-I00 and PID2023-146754NB-I00, funded by MCIU/AEI/10.13039/501100011033 and FEDER, EU, the project PCI2024-153478, funded by the European M-ERA.Net program, and the French National Research Agency through the ANR project NOCIME (New Observation and Control Issues Motivated by Epidemiology) for the period 2024-2027. In addition, author A.B. Kubik was supported by an FPU predoctoral fellowship and a mobility grant, both provided by the Ministry of Science, Innovation and Universities of the Spanish Government.


\bibliographystyle{abbrv}
\bibliography{mybibfile.bib}

\end{document}